\documentclass[preprint,12pt]{elsarticle}

\usepackage{graphicx}
\usepackage{amssymb,amsmath}
\usepackage{lineno}

\usepackage{amsfonts}
\usepackage{amsthm}
\usepackage{graphicx}
\usepackage[font=singlespacing]{caption}
\usepackage[normalem]{ulem}
\usepackage{nicefrac}
\usepackage{units}
\usepackage{rotating}
\usepackage{afterpage}
\usepackage{tikz}
\usetikzlibrary{arrows,positioning}
\usepackage{mathtools}
\usepackage[all,cmtip]{xy}

\DeclarePairedDelimiter{\floor}{\lfloor}{\rfloor}

\usepackage{fourier} 
\usepackage{array}
\usepackage{makecell}

\newtheorem{thm}{Theorem}[section]

\newtheorem{prop}[thm]{Proposition}
\newtheorem{cor}[thm]{Corollary}
\newtheorem{note}[thm]{Note}
\newtheorem{question}[thm]{Question}

\newtheorem{ex}[thm]{Example}
\newtheorem{rem}[thm]{Remark}
\newcommand{\R}{\mathbb{R}}
\newcommand{\C}{\mathbb{C}}
\newcommand{\Z}{\mathbb{Z}}
\newcommand{\Q}{\mathbb{Q}}

\newcommand{\notimplies}{%
  \mathrel{{\ooalign{\hidewidth$\not\phantom{=}$\hidewidth\cr$\implies$}}}}

\usepackage{csquotes}

\journal{Journal of Algebra}

\begin{document}

\begin{frontmatter}


\title{Atoms in Quasilocal Integral Domains}



\author{D.D. Anderson}
\ead{dan-anderson@uiowa.edu}

\author{K. Bombardier\corref{cor1}}
\ead{kevin-bombardier@uiowa.edu}

\cortext[cor1]{Corresponding author}
\address{Department of Mathematics, The University of Iowa, Iowa City, IA, 52242, USA}

\begin{abstract}
Let $(R,M)$ be a quasilocal integral domain.  We investigate the set of irreducible elements (atoms) of $R$.  Special attention is given to the set of atoms in $M \backslash M^2$ and to the existence of atoms in $M^2$.  While our main interest is in local Cohen-Kaplansky (CK) domains (atomic integral domains with only finitely many non-associate atoms), we endeavor to obtain results in the greatest generality possible.  In contradiction to a statement of Cohen and Kaplansky, we construct a local CK domain with precisely eight nonassociate atoms having an atom in $M^2$.
\end{abstract}


\end{frontmatter}

\section{Introduction}
\label{S:1}

Let $R$ be a (commutative) integral domain.  A nonzero nonunit $x \in R$ is \emph{irreducible}, or an \emph{atom}, if $x = ab$ implies $a$ or $b$ is a unit and $R$ is \emph{atomic} if each nonzero nonunit of $R$ is a finite product of atoms.  An atomic domain with only finitely many nonassociate atoms is called a \emph{Cohen-Kaplansky} (\emph{CK}) \emph{domain}.  (While a field is an atomic domain, even a CK domain, to avoid trivialities, we assume throughout that $R$ is not a field.)  While the purpose of this article is to study local CK domains and their atoms, in Section 2 we begin by investigating atoms in quasilocal domains that need not even be atomic.  While we focus on quasilocal domains, we should point out that the study of atoms or atomicity cannot generally be reduced to the quasilocal case.  Indeed, the ring of integer-valued polynomials is a two-dimensional Pr\"{u}fer BFD (and hence is atomic), but has a localization at a maximal ideal that is not atomic \cite[Example 2.7(b)]{AAZ}.  Conversely, if $R$ is a Bezout almost Dedekind domain that is not a PID (take $R = D(X)$ where $D$ is your favorite non-Dedekind almost Dedekind domain), then $R$ is not atomic, but each localization of $R$ is a DVR and hence atomic.  However, for CK domains we can effectively reduce to the local case, see Theorem \ref{Lattice}.  As usual two elements $a$ and $b$ of a domain $R$ are \emph{associates}, denoted $a \sim b$, if $b = ua$ for some unit $u \in R$.

The setup for Section 2 is a not necessarily atomic quasilocal domain $(R,M)$, usually with $M \neq M^2$.  (We reserve the term \enquote{local} for a Noetherian quasilocal domain.)  Set $\overline{R} = R/M$.  We begin by remarking that if $M^{\beta} = 0$ for some ordinal $\beta$, then $R$ satisfies ACCP (Theorem \ref{BFDgen}).  If $x \in M \backslash M^2$, $x$ is certainly an atom.  Special attention is given to the set of atoms contained in $M \backslash M^2$ and to the existence of atoms in $M^2$.  We say that $M^n$ is (\emph{weakly}) \emph{universal} if $M^n \subseteq Rx$ for each atom $x \in R$ ($x \in M \backslash M^2$).  We show that if there are exactly $n$ nonassociate atoms (in $M \backslash M^2$), then $M^{n-1} (M^n)$ is (weakly) universal (Theorem \ref{Weakly}). Suppose that $M \neq M^2$.  Let $\{x_{\alpha}\}_{\alpha \in \Lambda} \subseteq M \backslash M^2$ be a complete set of representatives of the one-dimensional $\overline{R}$-subspaces of $M / M^2$.  Then $\{x_{\alpha}\}_{\alpha \in \Lambda}$ is a set of nonassociate atoms of $R$ lying in $M \backslash M^2$ (thus we have a lower bound for the number of nonassociate atoms in $M \backslash M^2$) and $M^2$ is universal if and only if $\{x_{\alpha}\}_{\alpha \in \Lambda}$ is a complete set of nonassociate atoms of $R$ (lying in $M \backslash M^2$) (Theorem \ref{QLdoms}).  We show that $M^2$ is universal if and only if $[M:M] = \{x \in K | xM \subseteq M\}$ ($K$ the quotient field of $R$) is a quasilocal domain with principal maximal ideal $M$ (Theorem \ref{QLdoms}).  For $M^n$ universal $(n \geq 2)$, we give an upper bound for the number of nonassociate atoms in $M^{n-1} \backslash M^n$ (Theorem \ref{Upper}).  Finally we show that if $(R,M)$ is a quasilocal domain with $M \neq M^2$ having only finitely many nonassociate atoms, then $P = \bigcap_{n=1}^{\infty} M^n$ is prime and $R/P$ is a CK domain (Theorem \ref{Uni}).

Section 3 concentrates on local CK domains.  We review some characterizations of local CK domains.  We offer alternative proofs and sharpen several results from \cite{CK}.  Let $(R,M)$ be a local CK domain that is not a DVR.  Let $V = U([M:M]) / U(R)$ where for a ring $S$, $U(S)$ is the group of units of $S$.  Now $V$ is finite and $|V| \geq |\overline{R}|$ with $|V| \geq |\overline{R}|+1$ if $M$ is the maximal ideal of $[M:M]$ (Theorem \ref{Vcard}).  For $x \in M$ and $u \in U([M:M])$, $x \in M^{n-1} \backslash M^n$ $\iff$ $ux \in M^{n-1} \backslash M^n$ and $x$ is an atom $\iff$ $ux$ is an atom.  Thus the number of nonassociate atoms in $M^{n-1} \backslash M^n$ is a multiple of $|V|$, possibly $0$ for $n \geq 3$.  Moreover, $M^2$ is universal $\iff$ the nonassociate atoms consist of a single $V$-class $\iff$ the nonassociate atoms contained in $M \backslash M^2$ consist of a single $V$-class.  Thus if the number of nonassociate atoms in $R$ (or in $M \backslash M^2$) is prime, $M^2$ is universal (Corollary \ref{2p}).

The fourth section consists of examples.  Of particular interest are local CK domains of the form $R = K + WX + F[[X]]X^2$ where $K \subseteq F$ is an extension of finite fields and $W$ is a $K$-subspace of $F$.  Cohen and Kaplansky's paper \cite{CK} is entitled \enquote{Rings with a finite number of primes. I.}  (They use the term \enquote{prime} to mean an atom.)  II never appeared, but on page 472 in regard to the result on the universality of $M^{n-1}$ when $R$ is a CK domain with precisely $n$ nonassociate atoms, they state \enquote{This result will incidentally be considerably sharpened in the paper that follows.}  A question they raised, but were unable to answer, was whether a local CK domain $(R,M)$ could have an atom in $M^2$.  To quote from page $473$ of their paper: \enquote{Whether or not there exist rings with a prime (sic) in $M^2$ is a question that has not yet been settled.  It follows from (2), and the fact that $k$ and $N$ are at least $2$, that such a ring must have at least seven primes (sic).  Since we shall prove below that $M^2$ is universal when $n$ is prime, the lower bound becomes $n=8$.  We shall continue this discussion in the second paper; but we remark that at the moment our best result has ruled out the possibility of a prime (sic) in $M^2$ for $n=8$ or $n=9$.}  Now in \cite{AMO} it was shown that you can have an atom in $M^2$.  Using the construction given there, we give an example of a local CK domain $(R,M)$ with exactly eight nonassociate atoms having two nonassociate atoms in $M^2$.  Perhaps this is why II never appeared.  We also use the construction given in \cite{AMO} to construct a local CK domain $(R,M)$ with $M^{2n}$ universal, but $M^{2n-1}$ not universal.  In Section 5 we investigate the existence of local CK domains with exactly $n$ nonassociate atoms for small $n$.

\section{Atoms in Quasilocal Domains}

In this section we study the set of atoms of a quasilocal domain $(R,M)$.  We will usually assume that $M \neq M^2$ so we have atoms in $M \backslash M^2$.  While our main goal is to study local CK domains, in this section we try to keep the results as general as possible by not assuming that $R$ is atomic or that the number of nonassociate atoms involved is necessarily finite.  Several of the results of this section have previously been given for CK domains \cite{CK}.

Recall that $R$ is a \emph{bounded} \emph{factorization} \emph{domain} (\emph{BFD}) if for each nonzero nonunit $x \in R$ there is a natural number $N(x)$ so that if $x = x_1 \cdots x_n$ where $x_i \in R$ is a nonunit, then $n \leq N(x)$.  We say that $R$ satisfies the \emph{ascending chain condition} \emph{on principal ideals} (\emph{ACCP}) if any ascending chain of principal ideals of $R$ stabilizes.  It is well known and easily proved that

\begin{center}
BFD $\implies$ ACCP $\implies$ atomic
\end{center}

\noindent and that none of these implications can be reversed, even for quasilocal domains.  See Section 4 for more details.  We next generalize the well known result that a quasilocal domain $(R,M)$ with $\bigcap_{n=1}^{\infty} M^n = 0$ is a BFD.  Recall that $M^{\beta}$ is defined for each ordinal $\beta$ where $M^{\beta + 1} = M M^{\beta}$ and for $\beta$ a limit ordinal $M^{\beta} = \bigcap_{\alpha < \beta} M^{\alpha}$.

\begin{thm} \label{BFDgen}
Let $(R,M)$ be a quasilocal domain.  If $M^{\alpha}=0$ for some ordinal $\alpha$, then $R$ satisfies ACCP.  If further $M^{w} = \bigcap_{n=1}^{\infty} M^n = 0$, $R$ is a BFD.
\end{thm}

\begin{proof}
  Define a function $\phi: M \backslash \{0\} \to \text{ORD}$ by $\phi(x) = \beta$ where $x \in M^{\beta} \backslash M^{\beta+1}.$  Now for $x,y \in M \backslash \{0\}$, $\phi(xy) > \phi(x)$.  Hence if $0 \neq Rx_1 \subsetneq Rx_2 \subsetneq Rx_3 \subsetneq \cdots$ is an infinite ascending chain of principal ideals in $R$, $\phi(x_1) > \phi(x_2) > \phi(x_3) > \cdots$ is an infinite descending chain of ordinals, a contradiction.  (This is \cite[Proposition 2]{AJ}.)  For the case where $0 = M^{w} = \bigcap_{n=1}^{\infty} M^n$, let $0 \neq x = x_1 \cdots x_m$ where $x_i \in M$.  Then $m \leq \phi(x)$; so $R$ is a BFD.
\end{proof}

Let $S$ and $T$ be subsets of $R \backslash \{0\}$ where $R$ is an integral domain.  We say that $S$ is \emph{universally divisible by $T$} if each element of $S$ is divisible by each element of $T$, or equivalently, $S \subseteq \bigcap_{t \in T} Rt$.  For $(R,M)$ quasilocal, $M^n$ is (\emph{weakly}) \emph{universal} if $M^n$ is universally divisible by $T=\{x | x \in R \text{ is an atom}\}$ ($T=\{x | x \in M \backslash M^2$).  The concept of $M^n$ being universal was introduced by Cohen and Kaplansky \cite{CK} who characterized the CK domains with $M^2$ universal and showed that if $R$ is a local CK domain with exactly $n$ nonassociate atoms, then $M^{n-1}$ is universal; see Theorem \ref{Weakly} for a generalization.

%  See Theorem what?

We next characterize quasilocal domains $(R,M)$ with $M^2$ universal.

\begin{thm} \label{QLdoms}
Let $(R,M)$ be a quasilocal domain with $M \neq M^2$.  Put $\overline{R}=R/M$.  Let $\{V_{\alpha}\}_{\alpha \in \Lambda}$ be the set of one-dimensional $\overline{R}$-subspaces of $M/M^2$.  For each $\alpha \in \Lambda$, let $x_{\alpha} \in M \backslash M^2$ with $V_{\alpha} = \overline{R} \overline{x_{\alpha}}$.

\begin{itemize}
 \item[(1)]  If $x \in M \backslash M^2$, $x$ is an atom of $R$.
 
 \item[(2)]  If $x,y \in M$ with $x \sim y$, then $\overline{R} \overline{x} = \overline{R} \overline{y}$.  So associate atoms of $M \backslash M^2$ determine the same one-dimensional subspace of $M/M^2$.
 
 \item[(3)]  $\{x_{\alpha}\}_{\alpha \in \Lambda}$ is a set of nonassociate atoms in $M \backslash M^2$.
 
 \item[(4)]  Suppose that there is an atom $q \in M^2$.  Let $\{u_{\beta}\}_{\beta \in \Gamma}$ be a complete set of representatives of $\overline{R}$.  Then $\{x_{\alpha}+u_{\beta}q\}_{(\alpha,\beta)\in \Lambda \times \Gamma}$ is a set of nonassociate atoms in $M \backslash M^2$.
 
 \item[(5)]  The following are equivalent:
 
 \begin{itemize}
    \item[(a)]  $M^2$ is universal,
    
    \item[(b)]  $\{x_{\alpha}\}_{\alpha \in \Lambda}$ is a complete set of nonassociate atoms in $M \backslash M^2$,
    
    \item[(c)]  $\{x_{\alpha}\}_{\alpha \in \Lambda}$ is a complete set of nonassociate atoms of $R$,
    
    \item[(d)]  $aM=M^2$ for each $a \in M \backslash M^2$,
    
    \item[(e)]  $M^2$ is weakly universal, and
    
    \item[(f)]  $[M:M]$ is a quasilocal domain with principal maximal ideal $M$.
 \end{itemize}
 
 \item[(6)]  $R$ is an atomic domain with $M^2$ universal if and only if $[M:M]$ is a DVR with maximal ideal $M$.  In this case $R$ is even a BFD.
  
 \item[(7)]  Suppose that $R$ is local.  Then $M^2$ is universal if and only if $R'$, the integral closure of $R$, is a DVR with maximal ideal $M$.
\end{itemize}
\end{thm}

\begin{proof}
  (1) and (2) are clear and together prove (3).  (We note that (3) is well known with the finite case given in \cite{CK}.)  (4) Cohen and Kaplansky \cite{CK} proved this for $R$ a CK domain.  While their proof extends to this case mutatis mutandis, we give the simple proof for completeness.  Certainly each $x_{\alpha}+u_{\beta}q \in M \backslash M^2$ is an atom.  Suppose that $x_{\alpha}+u_{\beta}q \sim x_{\alpha '}+u_{\beta '}q$, so $x_{\alpha}+u_{\beta}q = u (x_{\alpha '}+u_{\beta '}q)$ for some unit $u \in R$.  Then $\overline{x_{\alpha}} = \overline{u} \; \overline{x_{\alpha '}}$, so $\alpha = \alpha '$ and $\overline{u} = \overline{1}$ in $\overline{R}$.  Now $x_{\alpha}(1-u) = (u u_{\beta '} - u_{\beta})q$, so $x_{\alpha} \not \sim q$ gives $u u_{\beta '} - u_{\beta} \in M$.  Finally, $\overline{u}=\overline{1}$ in $\overline{R}$ gives $\overline{u_{\beta '}} = \overline{u_{\beta}}$;  so $\beta ' = \beta$.
 
(5)  (a) $\implies$ (b)  Suppose that $M^2$ is universal.  Let $x \in M \backslash M^2$ be an atom.  So $\overline{R} \overline{x} = V_{\alpha} = \overline{R} \overline{x_{\alpha}}$ for some $\alpha \in \Lambda$.  Now $M^2$ universal gives $M^2 \subseteq Rx \cap Rx_{\alpha}$, so $Rx = Rx+M^2 = Rx_{\alpha}+M^2=Rx_{\alpha}$.  Hence $x \sim x_{\alpha}$.  (b) $\implies$ (c)  Suppose there is an atom $q \in M^2$.  Then by (4), for any $\alpha \in \Lambda$, $x_{\alpha}+q$ is an atom in $M \backslash M^2$ not associated with any $x_{\beta}$, a contradiction.  (c) $\implies$ (a)  Let $x \in M^2$.  For any $\alpha \in \Lambda$, $x_{\alpha}+x \in M \backslash M^2$ and hence is an atom.  So $x_{\alpha}+x \sim x_{\beta}$ for some $\beta$, necessarily with $\beta = \alpha$ since $\overline{x_{\alpha}}=\overline{x_{\beta}}$.  So $x_{\alpha}+x = ux_{\alpha}$ for some unit $u \in R$.  Then $x=(u-1)x_{\alpha} \in R x_{\alpha}$.  So $M^2 \subseteq \bigcap_{\alpha \in \Lambda} Rx_{\alpha} = \bigcap \{Ra \; | \; a \text{ is an atom of R}\}$.  (a) , (e) $\iff$ (d)  Just observe that for $a \in M \backslash M^2$, $Ra \supseteq M^2 \iff aM=M^2$.  (d) $\implies$ (f)  Let $x \in [M:M]$ be a nonunit, so $xM \subsetneq M$.  Let $a \in M \backslash M^2$.  Now $xa \in M \backslash M^2 \implies xaM=M^2 = aM \implies xM=M$, a contradiction.  Thus $xa \in M^2 = aM \implies x \in M$.  Hence $[M:M]$ is quasilocal with maximal ideal $M$.  Let $a \in M \backslash M^2$; we show $M = a[M:M]$.  For $b \in M$, $aM = M^2 \supseteq bM$ so $b/a \in [M:M]$.  Thus $b \in a[M:M]$.  Hence $M \subseteq a[M:M] \subseteq M$.  (f) $\implies$ (a)  Suppose $M=a[M:M]$ where $a \in M$.  So atoms of $R$ have the form $ua$ where $u \in [M:M]$ is a unit.  Hence $M^2 = a^2 [M:M]=a[M:M](ua) \subseteq Rua$.

(6)  $(\impliedby)$  Suppose that $[M:M]$ is a DVR with maximal ideal $M$.  By (5), $M^2$ is universal.  Since $[M:M]$ is a DVR, $\bigcap_{n=1}^{\infty} M^n = 0.$  Hence $R$ is a BFD and hence is atomic.  ($\implies$)  This is \cite[Corollary 5.2]{AMO}, but we offer a simple self-contained proof.  By (5), $[M:M]$ is a quasilocal domain with principal maximal ideal $M$.  We first show that $[M:M]$ is a valuation domain.  Let $x,y \in M \backslash \{0\}$ so $x = a_1 \cdots a_n$, $y = b_1 \cdots b_m$ where $a_i, b_j$ are atoms.  Now by (5) $a_i M = M^2 = b_j M$.  Hence $a_1 \cdots a_n M = M^{n+1}$ and $b_1 \cdots b_m M = M^{m+1}$.  Suppose $n \leq m$.  Then $yM = b_1 \cdots b_m M = M^{m+1} \subseteq M^{n+1} = a_1 \cdots a_n M = xM$.  So $y/x \in [M:M]$.  It follows that $[M:M]$ is a valuation domain.  Now $M^{n+1} \subseteq Rx$.  So $\bigcap_{n=1}^{\infty} M^n \subseteq \bigcap \{Rx \; | \; x \in M \backslash \{0\}\}=0$.  Thus $[M:M]$ is a DVR.

(7)  Suppose that $R$ is local.  Then $R \subseteq [M:M]\subseteq R'$.  If $M^2$ is universal, $[M:M]$ is a DVR by (6) and hence $[M:M] = R'$.  Conversely, suppose that $R'$ is a DVR with maximal ideal $M$.  Then $R'M \subseteq M$ so $R' \subseteq [M:M]$.  Hence $[M:M] = R'$ is a DVR with maximal ideal $M$.  By (6), $M^2$ is universal.
\end{proof}

\begin{cor}
Let $(V,M)$ be a quasilocal domain with principal maximal ideal $M$.  Let $L$ be a subfield of $V/M$.  Let $R$ be the pullback of

\begin{center}
\[\xymatrixcolsep{1.5pc}
\xymatrix{
R \ar@{-->}[r] \ar@{-->}[d] & V \ar[d] \\
L \ar@{}[r]|-*[@]{\subseteq} & V/M.}
\]
\end{center}

Then $R$ is a quasilocal domain with maximal ideal $M$, $M^2$ universal, and $V=[M:M]$.

Conversely, if $R$ is a quasilocal domain with maximal ideal $M \neq M^2$ and $M^2$ universal, then $V=[M:M]$ has principal maximal ideal $M$ and $R$ is the pullback of

\begin{center}
\[\xymatrixcolsep{3pc}
\xymatrix{
 & V \ar[d] \\
R/M \ar@{^{(}->}[r] & V/M.}
\]
\end{center}

\end{cor}

\begin{cor}
Let $(V,M)$ be a quasilocal domain with principal maximal ideal $M$.  Let $(D,P)$ be a quasilocal subring of $V$ with $P = D \cap M$.  Let $R = D+M$.  Then $R$ is a quasilocal domain with maximal ideal $M$ and $M^2$ is universal.
\end{cor}

%  Proof?

\begin{note}
 $(R,M)$ with $M^2$ universal does not imply that $R$ is atomic.  For example, take $(V,M)$ to be a valuation domain with principal maximal ideal $M$ and dim $V > 1$.  Let $(D,P)$ be a subring of $V$ with $P = M \cap D$ (e.g., $D = V$), then $R = D+M$ is quasilocal with maximal ideal $M$ and $M^2$ is universal, but $R$ is not atomic since $V = [M:M]$ is not a DVR.
\end{note}

The next theorem investigates the number of nonassociate atoms in $M \backslash M^2$ for a quasilocal domain $(R,M)$.  We need the following definitions.

Let $R$ be an integral domain.  We call $R$ a \emph{finite atom} (\emph{FA}) \emph{domain} if $R$ has only finitely many (possible none) nonassociate atoms.  In the extreme case where $R$ has no atoms, following \cite{CDM} we call $R$ an \emph{antimatter domain}.  Thus an atomic FA domain is just a CK domain.  For $(R,M)$ quasilocal, $R$ is a \emph{weak finite atom} (\emph{WFA}) \emph{domain} if there are only finitely many nonassociate atoms in $M \backslash M^2$.  Thus if $M=M^2$, $R$ is a WFA domain.  Let $(V,M)$ be a valuation domain.  As either $M=M^2$ and $V$ is antimatter or $M=(a)$ and $a$ is the only atom of $V$ up to associates, $V$ is a FA domain.

\begin{thm} \label{QLatom}
Let $(R,M)$ be a quasilocal domain.  Put $\overline{R} = R/M$.

\begin{itemize}
\item [(1)]  $R$ has no atoms in $M \backslash M^2$ if and only if $M=M^2$.

\item [(2)]  If there are only finitely many nonassociate atoms in $M \backslash M^2$, but at least one, then $M$ is finitely generated.  Thus if $R$ is a WFA domain, either $M=M^2$ or $M$ is finitely generated.

\item [(3)]  If $M = (a)$ is principal, then $a \in M \backslash M^2$ and $a$ is the only atom of $R$ up to associates.  Conversely, suppose that up to associates there is only one atom $a$ in $M \backslash M^2$.  Then $M = (a)$.

\item [(4)]  The following are equivalent:

\begin{itemize}
\item [(a)]  $R$ is a DVR,

\item [(b)]  $M$ is principal and $R$ is atomic,

\item [(c)]  $R$ is atomic, $\text{dim}_{\overline{R}} \; M/M^2 = 1$, and there are only finitely many atoms in $M \backslash M^2$ up to associates,

\item [(d)]  $R$ is atomic and has exactly one atom up to associates, and

\item [(e)]  $R$ is atomic and has exactly one atom in $M \backslash M^2$ up to associates.
\end{itemize}

\item [(5)]  If $\overline{R}$ is infinite, there are either no atoms in $M \backslash M^2$ (i.e., $M=M^2$), exactly one atom in $M \backslash M^2$ up to associates (i.e., $M$ is principal), or there are infinitely many nonassociate atoms in $M \backslash M^2$.  Thus if $R$ is a WFA domain, either $M=M^2$ or $M$ is principal.

\item [(6)]  Suppose that $\overline{R}$ is finite.  If $\text{dim}_{\overline{R}} M/M^2$ is infinite, there are infinitely many nonassociate atoms in $M \backslash M^2$.  If $\text{dim}_{\overline{R}} \; M/M^2 = 0$, then $M=M^2$ and there are no atoms in $M \backslash M^2$.  Suppose that $1 \leq \text{dim}_{\overline{R}} M/M^2 = k < \infty$ and $m:= (|\overline{R}|^k -1) / (|\overline{R}|-1)$.  Suppose there are $n$ nonassociate atoms in $M \backslash M^2$.  Then $n \geq m$.  If there is an atom in $M^2$, then $n \geq m |\overline{R}|$ and hence there are at least $m |\overline{R}|+1 = (|\overline{R}|^{k+1}-1)/(|\overline{R}|-1)$ nonassociate atoms in $R$.  Suppose $n < \infty$, then $M$ can be generated by $\floor{\text{log}_{|\overline{R}|} \; n}+1$ elements.  Finally, the following are equivalent:

\begin{itemize}
\item [(a)]  $M^2$ is universal,

\item [(b)]  n=m, and

\item [(c)]  there are exactly $m$ nonassociate atoms in $M \backslash M^2$.
\end{itemize}

%  (|\overline{R}^{k+1}|-1)/(|\overline{R}|-1) no overline on the second piece?

\item [(7)]  $R$ cannot have exactly two nonassociate atoms in $M \backslash M^2$.  If either $R$ is atomic or $M \neq M^2$, $R$ cannot have exactly two nonassociate atoms.  If $M \neq M^2$ and $R$ is not a DVR, there are at least three nonassociate atoms in $M \backslash M^2$.  If further there is an atom in $M^2$, $R$ has at least six nonassociate atoms in $M \backslash M^2$ and hence $R$ has at least seven nonassociate atoms.  (In the next section we will see that if a local CK domain has an atom in $M^2$, then there must be at least eight nonassociate atoms.  In Section 4 we give an example of a local CK domain with eight nonassociate atoms having an atom in $M^2$.)
\end{itemize}
\end{thm}

\begin{proof}
(1) Clear since an element of $M \backslash M^2$ is an atom.  (2) Now suppose that $a_1,\ldots,a_n$ ($n \geq 1)$ is a complete set of nonassociate atoms in $M \backslash M^2$.  Then $M = (a_1) \cup \cdots \cup (a_n) \cup M^2 = (a_1,\ldots,a_n) \cup M^2$.  Since $n \geq 1$, $M \neq M^2$, so $M = (a_1,\ldots,a_n)$.  (3) ($\implies$)  Clear ($\impliedby$)  By Theorem \ref{QLdoms}, $\text{dim}_{\overline{R}} M/M^2 = 1$, so $M = (a) + M^2$ for some $a \in M \backslash M^2$.  By (2), $M$ is finitely generated, so $M = (a)$ by Nakayama's Lemma.  (4) (a) $\implies$ (b) $\implies$ (c)  Clear.  (c) $\implies$ (d)  By (2), $M$ is finitely generated.  Then $\text{dim}_{\overline{R}} M/M^2 = 1$ gives $M$ is principal and hence $R$ has only one atom up to associates.  (d) $\implies$ (e)  Suppose $a$ is the only atom of $R$ up to associates.  Then $M = (a)$, so $a \notin M^2$.  (e) $\implies$ (a)  Let $a$ be the atom in $M \backslash M^2$.  As we have seen $M = (a)$.  So $a$ is the only atom of $R$ up to associates.  Hence every nonzero nonunit of $R$ has the form $u a^n$ where $u$ is a unit and $n \geq 1$.  Thus $R$ is a DVR.

(5)  Suppose that $\overline{R}$ is infinite.  If $\text{dim}_{\overline{R}} M/M^2 > 1$, then $M/M^2$ has infinitely many one-dimensional subspaces and hence there are infinitely many nonassociate atoms in $M \backslash M^2$ by Theorem \ref{QLdoms} (3).  If $\text{dim}_{\overline{R}} M/M^2 = 1$ and there are only finitely many nonassociate atoms in $M \backslash M^2$, then $M$ is principal since $M$ is finitely generated by (2).  So up to associates there is one atom in $M \backslash M^2$.  If $\text{dim}_{\overline{R}} M/M^2 = 0$, $M = M^2$ and there are no atoms in $M \backslash M^2$.  The last statement is now immediate.

(6)   Suppose that $\overline{R}$ is finite.  The statements concerning the cases when $\text{dim}_{\overline{R}} M/M^2$ is infinite or $0$ are clear.  So suppose that $1 \leq \text{dim}_{\overline{R}} M/M^2 = k < \infty$.  Then $M/M^2$ has $m:= (|\overline{R}|^k-1) / (|\overline{R}|-1)$ one-dimensional $\overline{R}$-subspaces.  Thus by Theorem \ref{QLdoms} (3), there are at least $m$ nonassociate atoms in $M \backslash M^2$.  Suppose there is an atom $q$ in $M^2$.  Let $x_1,\ldots,x_m$ be the $m$ atoms corresponding to the one-dimensional subspaces of $M/M^2$ and $u_1,\ldots,u_{|\overline{R}|}$ be a complete set of representatives of $\overline{R}$.  Then by Theorem \ref{QLdoms} (4), the elements $x_i + u_j q$, $i = 1,\ldots,m$, $j=1,\ldots,|\overline{R}|$ are $m |\overline{R}|$ nonassociate atoms in $M \backslash M^2$.  Thus there are at least $m |\overline{R}|+1 = (|\overline{R}|^{k+1}-1)/(|\overline{R}|-1)$ nonassociate atoms.  Suppose $n<\infty$, then $M$ is finitely generated.  Now $M$ can be generated by $k$ elements and $(|\overline{R}|^{k}-1)/(|\overline{R}|-1) \leq n$.  Hence $k \leq \floor{\text{log}_{|\overline{R}|} \; n}+1$.  The statement concerning the universality of $M^2$ follows from Theorem \ref{QLdoms} (5).  (Most of (5) is given in \cite{CK} for the case where $R$ is a local CK domain.  However, the hypothesis that $R$ is atomic is not needed.)

(7)  The fact that there cannot be exactly two nonassociate atoms in $M \backslash M^2$ follows from (5) and (6).  Suppose that $R$ has exactly two nonassociate atoms $p$ and $q$.  First suppose that $R$ is atomic.  Then $p+q$ must have an atomic factor not an associate of $p$ or $q$, a contradiction.  Next assume that $M \neq M^2$.  So there is an atom in $M \backslash M^2$, say $p \in M \backslash M^2$.  If $q \in M \backslash M^2$ we contradict the first statement of (7).  So $q \in M^2$.  But then by Theorem \ref{QLdoms} (4), $p+q$ is a third atom.  The last statement follows from (6).
\end{proof}

However, as the following example shows, it is quite possible to have infinitely many nonassociate atoms in $M \backslash M^2$ where $(R,M)$ is a quasilocal domain with $\overline{R}=R/M$ finite and $\text{dim}_{\overline{R}} M / M^2 = 1$.

\begin{ex}
Let $(V,N)$ be a quasilocal domain with nonzero idempotent maximal ideal $N$ (e.g., $V$ is a valuation domain with nonprincipal maximal ideal).  Let $R = V[[X]]$.  So $R$ is a quasilocal domain with maximal ideal $M=(X,N)$ and $\overline{R}=R/M=V/N$.  Here $\text{dim}_{\overline{R}} M/M^2 = 1$, but $M$ is not principal and there are infinitely many nonassociate atoms in $M \backslash M^2$ (since for $n,n' \in N, n+X \sim x'+X \implies n \sim n'$).  By choosing $V/N$ to be finite, we even have that $\overline{R}$ and $M / M^2$ are finite.
\end{ex}

We have given a lower bound, the cardinality of the set of one-dimensional $\overline{R}$-subspaces of $M/M^2$, for the number of nonassociate atoms in $M \backslash M^2$.  We next give an upper bound, the cardinality of the set of one-dimensional $\overline{R}$-subspaces of $M^{n-1} / M^n$, for the number of nonassociate atoms in $M^{n-1} \backslash M^n$ where $M^n$ is universal.

\begin{thm} \label{Upper}
Let $(R,M)$ be a quasilocal domain, $\overline{R}=R/M$, and $n \geq 2$.  Let $x,y \in M^{n-1} \backslash M^n$.

\begin{itemize}
\item[(1)]  If $x \sim y$, then $\overline{R} \overline{x} = \overline{R} \overline{y}$.

\item[(2)]  Suppose that $y$ is an atom and $M^n$ is universal.  If $\overline{R} \overline{x} = \overline{R} \overline{y}$, then $x \sim y$.

\item[(3)]  Suppose that $M^n$ is universal.  Then two atoms $x,y \in M^{n-1}$ are associates if and only if $\overline{R} \overline{x} = \overline{R} \overline{y}$.

\item[(4)]  Suppose that $M^n$ is universal.  Let $\alpha$ be the cardinality of the set of one-dimensional $\overline{R}$-subspaces of $M^{n-1} / M^n$.  Then there are at most $\alpha$ nonassociate atoms in $M^{n-1} \backslash M^n$.  Hence if $\overline{R}$ and $l := \text{dim}_{\overline{R}} \; M^{n-1} / M^n$ are finite, there are at most $(|\overline{R}|^l -1) / (|\overline{R}|-1)$ nonassociate atoms in $M^{n-1} \backslash M^n$. For $n \geq 3$, there are at most $(|\overline{R}|^l-1) / (|\overline{R}|-1)-1$ nonassociate atoms in $M^{n-1} \backslash M^n$.
% \item[(5)]  Suppose that $M^n$ is universal.  If there are only finitely many nonassociate atoms in $M \backslash M^2$, then there are only finitely many nonassociate atoms in $M^{n-1} \backslash M^n$.

\item[(5)]  Suppose that $M^n$ is universal.  Let $\{x_{\alpha}\}_{\alpha \in \Lambda}$ be a complete set of representatives for the one-dimensional $\overline{R}$-subspaces of $M^{n-1} / M^n$.  Then each $x_{\alpha}$ is an atom if and only if $n=2$ or $M^{n-1}=M^n$.
\end{itemize}
\end{thm}

\begin{proof}
(1) Clear.  (2) Now $\overline{R} \overline{x} = \overline{R} \overline{y}$ gives $x -ry \in M^n$ for some $r \in R$.  Since $M^n$ is universal, $y | x-ry$, so $y | x$.  Since $x,y \in M^{n-1} \backslash M^n$, $x \sim y$.  (3)  This follows from (1) and (2).  (4) The first part follows from (3).  Suppose that $n \geq 3$.  Then by (5) not all of the one-dimensional $\overline{R}$-subspaces of $M^{n-1} / M^n$ can give rise to an atom.
% (5)  Suppose that there are only finitely many nonassociate atoms in $M \backslash M^2$.  If $M=M^2$, $M^{n-1} = M^n$, and the result holds.  So suppose that $M \neq M^2$.  If $\overline{R}$ is infinite, $M$ is principal.  Then $n=1$ and the result holds.  Suppose that $\overline{R}$ is finite.  Then $M$ is finitely generated so $\text{dim}_{\overline{R}} \; M^{n-1} / M^n$ is finite.  The result follows.

(5)  If $n=2$, then each $x_{\alpha} \in M \backslash M^2$ and hence is an atom.  If $M^{n-1}=M^n$, the result is obvious.  Conversely, suppose that each $x_{\alpha}$ is an atom.  Then for $y \in M^{n-1} \backslash M^n$, $\overline{R} \overline{y} = \overline{R} \overline{x_{\alpha}}$ for some $\alpha \in \Lambda$.  Hence by (2), $y \sim x_{\alpha}$ and hence is an atom.  Thus each element $y \in M^{n-1} \backslash M^n$ is an atom.  Suppose that $n > 2$.  Let $x \in M \backslash M^2$.  If $xM^{n-2} \not \subseteq M^n$, we have $xm \in M^{n-1} \backslash M^n$ for some $m \in M$.  But this is a contradiction since $xm$ is not an atom.  Thus $xM^{n-2} \subseteq M^n$ for each $x \in M$ and hence $M^{n-1} = M M^{n-2} \subseteq M^n$.  So $M^{n-1} = M^n$.
\end{proof}

Cohen and Kaplansky \cite{CK} showed that if $(R,M)$ is a local CK domain with precisely $n$ nonassociate atoms, then $M^{n-1}$ is universal.  We generalize this result in Theorem \ref{Weakly} (which does not require $R$ to be atomic).  Our proof of Theorem \ref{Weakly} is modeled after their proof.  But we first show that for $(R,M)$ a quasilocal (W)FA domain with $M \neq M^2$, some power of $M$ is (weakly) universal.  We also generalize the well known result that if $P$ is a principal prime ideal, then $Q = \bigcap_{n=1}^{\infty} P^n$ is prime and there are no prime ideals properly between $P$ and $Q$.

\begin{thm} \label{Uni}
Let $(R,M)$ be a quasilocal domain and let $P = \bigcap_{n=1}^{\infty} M^n$.

\begin{itemize}
\item[(1)]  Suppose that $M \neq M^2$.  Then $R$ is a (W)FA domain if and only if either (a) $M$ is principal, or (b) $R/M$ is finite, $M$ is finitely generated, and some power of $M$ is (weakly) universal.

\item[(2)]  If $R$ is a WFA domain with $M \neq M^2$, then there can be no atoms in $P$ and each nonzero nonunit of $R$ has an atom as a factor.  For $a \in M \backslash M^2$, $(a)$ is $M$-primary.  Thus if $R$ is Noetherian, dim $R=1$.

\item[(3)]  Let $R$ be an FA domain.  Then $R/P$ is a field or CK domain and hence $P$ is prime and there are no prime ideals properly between $P$ and $M$.  If $M \neq M^2$ and either dim $R=1$ or $R$ is completely integrally closed, $R$ is a CK domain.
\end{itemize}
\end{thm}

\begin{proof}
(1) ($\impliedby$)  If $M$ is principal, then $R$ is a FA domain and $M$ is universal.  So suppose that $R/M$ is finite, $M$ is finitely generated, and $M^l$ is (weakly) universal.  Since $R/M$ is finite and $M$ is finitely generated, $R/M^l$ is finite.  Thus there are only finitely many principal ideals $(a) \supseteq M^l$.  Hence $R$ is a (W)FA domain.  ($\implies$)  Since $M \neq M^2$, there is at least one atom in $M \backslash M^2$.  By Theorem \ref{QLatom} (2), $M$ is finitely generated.  If $R/M$ is infinite, Theorem \ref{QLatom} (5) gives that $M$ is principal.  So suppose that $M$ is not principal.  Then $R/M$ is finite.  Let $a_1,\ldots,a_n$ be a complete set of nonassociate atoms (in $M \backslash M^2$).  For the WFA case, $M = (a_1) \cup \cdots \cup (a_n) \cup M^2$ is an irredundant union.  So by McCoy's Theorem \cite{M}, there exists an $l$ with $M^l \subseteq (a_1) \cap \cdots \cap (a_n) \cap M^2 \subseteq (a_1) \cap \cdots \cap (a_n).$  So $M^l$ is weakly universal.  Now consider the FA case.  Since $R$ is a WFA domain with $M \neq M^2$, for $a \in M \backslash M^2$, some $M^l \subseteq (a)$.  So without loss of generality we can assume that $\bigcap_{n=1}^{\infty} M^n \subset (a_1)$.  Now each element of $M^i \backslash M^{i+1}$ for $i \geq 1$ is a finite product of atoms.  Hence $M = (a_1) \cup \cdots \cup (a_n)$ and this union is irredundant.  So again by McCoy's Theorem, some $M^k \subseteq (a_1) \cap \cdots \cap (a_n)$.  So $M^k$ is universal.

(2) The first statement follows since some $M^l$ is weakly universal so each nonzero element of $P$ has an atom as a proper factor and since each element of $M^n \backslash M^{n+1}$ is a finite product of atoms. Let $a \in M \backslash M^2$.  We noted in the proof of (1) that some $M^l \subseteq (a)$.  Thus (a) is $M$-primary.  If $R$ is Noetherian, the Principal Ideal Theorem gives that dim $R=1$.

(3) We can assume that $M \neq M^2$.  Since $M$ is finitely generated, the powers of $M$ properly descend.  Thus $\overline{R} = R/P$ is not Artinian.  Let $x_1, \ldots, x_n$ be a complete set of nonassociate atoms of $R$.  Then every nonzero nonunit of $\overline{R}$ is a unit of $\overline{R}$ times a power-product of the $\overline{x_i}$'s.  By \cite[Theorem 1]{A} $\overline{R}$ is either a finite ring, SPIR, or CK domain.  Since in the first two cases $\overline{R}$ is Artinian, we must have that $\overline{R}$ is a CK domain.  Thus $P$ is prime and there are no prime ideals properly between $P$ and $M$.  Suppose that $M \neq M^2$.  If dim $R=1$, $P=0$ and $R$ is a CK domain.  Suppose that $R$ is completely integrally closed.  Then for $a \in M$, $\bigcap_{n=1}^{\infty} (a^n)=0$.  Let $a \in M \backslash M^2$, so some $M^l \subseteq (a)$.  Then $P = \bigcap_{n=1}^{\infty} M^n \subseteq \bigcap_{n=1}^{\infty} (a^n)=0$.  Thus $R$ is a CK domain (even a DVR).  %  Suppose that $M \neq M^2$ before If dim R = 1?
\end{proof}

\begin{rem}
Of course it is quite possible for a (W)FA domain $(R,M)$ to have $\bigcap_{n=1}^{\infty} M^n \neq 0$.  Let $(V,M)$ be a valuation domain, so either $M=M^2$ or $M$ is principal, and $\bigcap_{n=1}^{\infty} M^n = 0$ $\iff$ $V$ is a DVR.  We remark that we know of no quasilocal atomic domain $(R,M)$ with $M=M^2$.
\end{rem}

\begin{thm} \label{Weakly}
Let $(R,M)$ be a quasilocal domain.

\begin{itemize}
\item[(1)]  Suppose that there are exactly $n$, $0 \leq n < \infty$ nonassociate atoms in $M \backslash M^2$.  Then $M^n$ is weakly universal.

\item[(2)]  Suppose that $M \neq M^2$ and that there are exactly $2 \leq n < \infty$ nonassociate atoms in $R$.  Then $M^{n-1}$ is universal.
\end{itemize}
\end{thm}

\begin{proof}
Note that (1) is trivial for $n=0$. For $n=1$, $M$ is principal by Theorem \ref{QLatom} and (1) also holds.  So we can assume in both cases that $2 \leq n < \infty$.  Let $a_1, \ldots, a_n$ ($n \geq 2$) be a complete set of nonassociate atoms (in $M \backslash M^2$).  By the proof of Theorem \ref{QLatom} (2), $M = (a_1,\ldots,a_n)$; so $M^k = \left(\prod_{j=1}^k a_{i_j} \right)$.  Suppose that $M^{n-1}$ ($M^n$) is not (weakly) universal.  Without loss of generality we can assume that $a_1 \nmid \prod_{j=1}^{n-1} a_{i_j}$ $\left(a_1 \nmid \prod_{j=1}^n a_{i_j} \right)$.  Put $x_k := \prod_{j=k}^{n-1} a_{i_j} \left(\prod_{j=k}^n a_{i_j}  \right)$ for $k=1,\ldots,n-1$.  So $a_1 \nmid x_k$.  Set $y_k = a_1 + x_k$.  So $a_1 \nmid y_k$ and $a_{i_{n-1}} \nmid y_k$  In case (1) each $y_k \in M \backslash M^2$ and hence is an atom.  In case (2) each $y_k$ is divisible by an atom as noted in Theorem \ref{Uni}. Thus in either case (1) or (2) each of the $n-1$ elements $y_1,\ldots,y_{n-1}$ is divisible by one of the $n-2$ atoms $a_2,\ldots,a_{i_{n-2}}$.  So by the Pigeonhole Principle, we have $i,j$, $1 \leq i \leq j \leq n-1$ and $l$, $2 \leq l \leq i_{n-2}$ with $a_l | y_i,y_j$.  So $a_l | y_j - y_i = x_j (1 - x_i / x_j).$   Now $1-x_i/x_j$ is a unit, so $a_l | x_j$.  Hence $a_l | y_j - x_j = a_1$, a contradiction.
\end{proof}

Three remarks concerning Theorem \ref{Weakly} (2) are in order.  First $n-1$ may be the best possible.  For example $R = \text{GF}(2) + \text{GF}(2^2)[[X]]X$ (Example \ref{Wint}) (resp., $R = \text{GF}(2) + \text{GF}(2)[[X]]X^2$ (Example \ref{Corrected})) has $3$ (resp., $4$) nonassociate atoms and $M^2$ (resp., $M^3$) is universal while $M$ (resp., $M^2$) is not.  Second, a CK domain $(R,M)$ with $M^2$ universal can have an arbitrarily large number of nonassociate atoms.  So certainly $n-1$ need not be the least power of $M$ that is universal.  Third, the least power of $M$ that is universal can be arbitrarily large.  For each $n \geq 2$, Example \ref{AndMottEx} gives a local CK domain $(R_n,M_n)$ with $M_n^{2n}$ universal, but $M_n^{2n-1}$ not universal.

\section{Local CK Domains}

In this section we sharpen and offer alternative proofs for some of the results in \cite{CK} concerning the number of atoms in a local CK domain.  For the reader's convenience, we recall some characterizations of local CK domains.

\begin{thm} \label{CKequiv}
Let $(R,M)$ be a quasilocal domain and $\overline{R}=R/M$.  Then the following conditions are equivalent.

\begin{itemize}
\item[(1)]  $R$ is a CK domain.

\item[(2)]  Either $\overline{R}$ is infinite and $R$ is a DVR or $\overline{R}$ is finite and $R$ is a one-dimensional analytically irreducible local domain.

\item[(3)]  Either $\overline{R}$ is infinite and $R$ is a DVR or $\overline{R}$ is finite, $R'$ is a DVR and a finitely generated $R$-module where $R'$ is the integral closure of $R$.

\item[(4)]  Either $\overline{R}$ is infinite and $R$ is a DVR or $R$ is atomic (e.g., $R$ is Noetherian), $\overline{R}$ is finite, $M$ is finitely generated (e.g., $R$ is Noetherian), and some power of $M$ is universal.

\item[(5)]  $R$ is a one-dimensional local domain that is an FFD and has finite elasticity $\rho (R)$.

\item[(6)]  $R$ has group of divisibility $G(R) \cong \Z \oplus F$ where $F$ is finite.
\end{itemize}
\end{thm}

\begin{proof}
The equivalence of (1) - (3) may be found in \cite[Theorem 4.3]{AMO}.  (1) $\iff$ (4)  This follows from Theorem \ref{Uni} (1).  (1) $\iff$ (6) \cite[Corollary 3.6]{AMO}  (2) $\iff$ (5)  Let $(R,M)$ be a one-dimensional local domain.  Then $R$ is an FFD $\iff$ $R$ is a DVR or $\overline{R}$ is finite \cite[Corollary 6]{AMU}, and $\rho (R) < \infty$ $\iff$ $R$ is analytically irreducible \cite[Theorem 2.12]{AA}.
\end{proof}

Let $(R,M)$ be a local CK domain.  Since $R$ is analytically irreducible, the map $\theta: L(R) \to L(\hat{R})$ given by $\theta(I) = \hat{R} I$ where $L(R)$ (resp., $L(\hat{R})$) is the lattice of ideals of $R$ (resp., $\hat{R}$) and $\hat{R}$ is the $M$-adic completion of $R$, is a multiplicative lattice isomorphism.  So $R$ is a CK domain if and only if $\hat{R}$ is a CK domain and both the ideal structure and the factorization structure (up to units) of $R$ and $\hat{R}$ are identical.  Thus in the local case, very little is lost by assuming that $R$ is complete.  The following result \cite[Theorem 4.5]{AMO} characterizes complete local CK domains.

\begin{thm}
(1)  Let $F_0 \subseteq F$ be finite fields and let $n \geq 1$.  Suppose that $R$ is an integral domain with $F_0 + F[[X]]X^n \subseteq R \subseteq F[[X]]$.  Then $R$ is a complete local CK domain with residue field between $F_0$ and $F$.

Conversely, suppose that $(R,M)$ is a complete local CK domain with $R/M$ finite and char $R= \text{char} R/M$.  Let $F_0$ (resp., $F$) be a coefficient field for $R$ (resp., $R'$, the integral closure of $R$).  Then there exists an $n \geq 1$ with $F_0 + F[[X]]X^n \subseteq R \subseteq F[[X]]$.

% (2) Let $p > 0$ be prime.  Let $(S,Sp)$ be a Cohen algebra with quotient field $K$ and finite residue field.  Let $L$ be a finite field extension of $K$ and let $(D,D \pi)$ be the integral closure of $D$ in $L$, so $D$ is a complete DVR.  Let $n \geq 1$.  Let $R$ be a ring with $S + D \pi^n \subseteq R \subseteq D$.  Then $R$ is a complete local CK domain with integral closure $D$.

% Conversely, suppose that $R$ is a complete local CK domain, not a DVR, with char $R=0$ and (finite) residue field of characteristic $p > 0$.  Let $D$ be the integral closure of $R$, so $D$ is a complete DVR, say with maximal ideal $D \pi$.  Then there exists a Cohen algebra $(S,Sp)$ with $R$ (or equivalently, $D$) a finitely generated $S$-module and an $n \geq 1$ with $S + D \pi^n \subseteq R \subseteq D$.

(2)  Let $p > 0$ be prime and $\Z_p$ the $p$-adic integers and $\Q_p$ the field of rational $p$-adics, and let $L$ be a finite field extension of $\Q_p$.  Let $(\overline{\Z_p}, (\pi))$ be the integral closure of $\Z_p$ in $L$.  So $\overline{\Z_p}$ is a complete DVR.  Suppose that $R$ is an integral domain with $\Z_p + \pi^n \overline{\Z_p} \subseteq R \subseteq \overline{\Z_p}$ for some $n$.  Then $(R,M)$ is a complete local CK domain with char $R=0$ and char $R/M = p > 0$.

Conversely, suppose that $(R,M)$ is a complete local CK domain with char $R=0$ and $R/M$ finite with char $R/M = p > 0$.  Let $L$ be the quotient field of $R$ and $\overline{\Z_p}$ the integral closure of $\Z_p$ in $L$.  Then $\overline{\Z_p} = \overline{R}$ and there exists an $n \geq 1$ so that $\Z_p + \pi^n \overline{\Z_p} \subseteq R \subseteq \overline{\Z_p}$.
\end{thm}

Let $(R,M)$ be a local CK domain that is not a DVR.  Let $R'$ be the integral closure of $R$, so $R'$ is a DVR.  Since $M$ has grade one, $R \subsetneq M^{-1}$, so $R \subsetneq M^{-1} \subseteq [M:M] \subseteq R'$.  Since $R'$ is local, $[M:M]$ is as well.  Now $[M:M]$ local and $R \subsetneq [M:M]$ gives that $U(R) \subsetneq U([M:M])$.  So $V:= U([M:M]) / U(R)$ is a nontrivial subgroup of $U(R') /U(R)$ and $U(R') /U(R)$ is finite.  (The fact that $U(R')/U(R)$ (and hence $V$) is finite follows from \cite[Theorem 3.9]{GM}.  However, the fact that $V$ is finite also follows from the correspondence below.)  Fix $x \in M \backslash \{0\}$.  Then the set $\{R \sigma x \; | \; \sigma \in U([M:M])\}$ of principal ideals corresponds to the set $Vx$ via $R \sigma x \leftrightarrow \sigma x U(R)$.  But $\{R \sigma x | \sigma \in U([M:M])\}$ corresponds to a complete set of nonassociate elements of the form $\sigma x$ where $\sigma \in U([M:M])$.  Cohen and Kaplansky showed that $|V| \geq |\overline{R}|$ where $\overline{R} = R/M$.  (See the paragraph preceding \cite[Theorem 11]{CK}.)  We sharpen this result and give an alternative proof.

\begin{thm} \label{Vcard}
Let $(R,M)$ be a local CK domain that is not a DVR.  Let $\overline{R} = R/M$.  Then $|V| \geq |\overline{R}|$.  If $M$ is the maximal ideal of $[M:M]$, then $|V| \geq |\overline{R}|+1$.
\end{thm}

\begin{proof}
First, suppose $M \neq \mathcal{M}$, the maximal ideal of $[M:M]$.  Let $m \in \mathcal{M} \backslash M$.  Let $u_0 = 0, u_1 = 1, \ldots, u_{N-1}, N = |\overline{R}|$, be a complete set of representatives of $\overline{R}$.  So $u_1,\cdots,u_{N-1}$ are units of $R$.  Thus $1, u_1+m, \cdots,u_{N-1}+m$ are units of $[M:M]$ with $(u_i+m)U(R) \neq 1 U(R)$.  Suppose that $u_i + m = \lambda (u_j+m)$ for some $\lambda \in U(R)$.  Then $u_i-\lambda u_j = (\lambda-1)m \in R \cap \mathcal{M}=M$.  So $m \notin M$ gives $\lambda-1 \in M$; thus $\overline{\lambda}=\overline{1}$ in $\overline{R}$.  But then $0 \equiv u_i - \lambda u_j \equiv u_i - u_j$ mod $M$, so $i=j$.  Thus $1 U(R),(u_1+m)U(R),\ldots,(u_{N-1}+m)U(R)$ are $N$ distinct elements of $V$.

Now suppose that $M$ is the maximal ideal of $[M:M]$.  Now $[M:M]/M$ is an $\overline{R}$-vector space of dimension greater than one, so it has at least $N+1$ one-dimensional $\overline{R}$-subspaces.  Suppose that $\overline{R} \overline{v_1}, \cdots, \overline{R} \overline{v_l}$, $l \geq N+1$, are the one-dimensional $\overline{R}$-subspaces of $[M:M]/M$ where $v_i \in [M:M]$.  Since $M$ is the maximal ideal of $[M:M]$, $v_i \notin M$; so $v_i \in U([M:M])$.  Also, if $v_i = \lambda v_j$ for some $\lambda U(R)$, then $\overline{R} \overline{v_i} = \overline{R} \overline{v_j}$, and hence $i=j$.  So $|V| \geq l \geq N+1$.
\end{proof}

\begin{thm} \label{AtomEquiv}
Let $(R,M)$ be a local CK domain that is not a DVR.  Let $x \in M \backslash \{0\}$ and $\sigma \in U([M:M])$.

\begin{itemize}
\item[(1)]  $x$ is an atom if and only if $\sigma x$ is an atom.

\item[(2)]  $x \in M^n$ if and only if $\sigma x \in M^n$.

\item[(3)]  $x$ is an atom in $M^{n-1} \backslash M^n$ ($n \geq 2$) if and only if $\sigma x$ is an atom in $M^{n-1} \backslash M^n$.  Thus the number of nonassociate atoms of $R$ and the number of nonassociate atoms in $M \backslash M^2$ (in $M^{n-1} \backslash M^n \text{ for } n \geq 3$) is a nonzero multiple of $|V|$ (is a multiple of $|V|$, possibly $0$).

\item[(4)]  The following are equivalent.
  \begin{itemize}
  \item[(a)]  $M^2$ is universal.
  
  \item[(b)]  The atoms of $R$ are given by a single coset of $V$.
  
  \item[(c)]  The atoms of $M \backslash M^2$ are given by a single coset of $V$.
  \end{itemize}
\end{itemize}
\end{thm}

\begin{proof}
(1) Suppose that $x$ is an atom.  If $\sigma x = ab$ where $a,b \in M$, then $x = (\sigma^{-1}a)b$, a contradiction.  So $\sigma x$ is an atom.  Hence if $\sigma x$ is an atom, so is $x = \sigma^{-1} (\sigma x)$.
(2) Suppose $x \in M^n$ so $x = \sum_{i=1}^n m_{i_1} \cdots m_{i_n}$ for some $m_{i_j} \in M$.  Then $\sigma x = \sum_{i=1}^n (\sigma m_{i_1}) m_{i_2} \cdots m_{i_n} \in M^n$.  Thus if $\sigma x \in M^n$, $x = \sigma^{-1} (\sigma x) \in M^n$.

(3) This follows from (1) and (2) and the remarks concerning $V$ given in the paragraph preceding Theorem \ref{Vcard}.

(4) $(a) \implies (b)$ Suppose that $M^2$ is universal.  By Theorem \ref{QLdoms} (5), $R' = [M:M]$ is a DVR with maximal ideal $M=\pi [M:M]$ for any $\pi \in M \backslash M^2$.  Thus the atoms of $R$ have the form $\sigma \pi$ where $\sigma \in U([M:M])$.  So the single coset $V \pi$ gives the atoms of $R$.  $(b) \implies (c)$  Clear.  $(c) \implies (a)$ We have that for any $x \in M \backslash M^2$, the elements of $M \backslash M^2$ have the form $\sigma x$ where $\sigma \in U([M:M])$.  So for $a,b \in M \backslash M^2$, $a = u_1 x$ and $b = u_2 x$ for some $u_1,u_2 \in U([M:M])$.  So $ab = (u_1 x) (u_2 x) = (u_1 u_2 x) x \in MRx$.  Thus $M^2 \subseteq MRx + M^3$.  By Nakayama's Lemma $M^2 = MRx$.  By Theorem \ref{QLdoms} (5), $M^2$ is universal.
\end{proof}

Parts of Theorem \ref{AtomEquiv} were proved by Cohen and Kaplansky \cite{CK}.  They noted (1), $4 (a) \iff 4(b)$, and that the number of nonassociate atoms is a multiple of $|V|$.

\begin{cor} \label{2p}
Let $(R,M)$ be a local CK domain that is not a DVR.

\begin{itemize}
\item[(1)]  \cite[Corollary, page 475]{CK}  Suppose that the number of nonassociate atoms of $R$ is prime, then $M^2$ is universal.

\item[(2)]  Suppose that the number of nonassociate atoms in $M \backslash M^2$ is prime, then $M^2$ is universal.

\item[(3)]  Suppose that $R$ has exactly $2 p$ nonassociate atoms where $p$ is prime and $|\overline{R}| \neq 2$ where $\overline{R} = R/M$.  Then $R$ has no atoms in $M^2$.
\end{itemize}
\end{cor}

\begin{proof}
(1) and (2) follow immediately from Theorem \ref{AtomEquiv} (4).

(3)  Here $|V| \geq |\overline{R}| \geq 3$ by Theorem \ref{Vcard} and $|V| \; | \; 2p$, so $|V|=p$ or $2p$.  If $|V|=2p$, every atom is in $M \backslash M^2$ (in fact, $M^2$ is universal).  So suppose that $|V|=p$.  So there are either $2p$ nonassociate atoms in $M \backslash M^2$ or $p$ nonassociate atoms in $M \backslash M^2$ and $p$ nonassociate atoms in $M^2$.  In the first case every atom is in $M \backslash M^2$.  The second case cannot occur since by (2) if the number of nonassociate atoms in $M \backslash M^2$ is prime, $M^2$ is universal and hence all atoms lie in $M \backslash M^2$.
\end{proof}

\begin{cor}
Let $(R,M)$ be a local CK domain that is not a DVR and let $\overline{R} = R/M$.  Suppose that there are less than $2|\overline{R}|$ nonassociate atoms in $M \backslash M^2$.  Then $R$ has exactly $|\overline{R}|+1$ nonassociate atoms and $M^2$ is universal.
\end{cor}

\begin{proof}
If $M^2$ is not universal, then there are at least two cosets of $V$ containing atoms in $M \backslash M^2$, so there are at least $2 |V| \geq 2 |\overline{R}|$ nonassociate atoms in $M \backslash M^2$.  Thus $M^2$ must be universal.  Hence $R$ has $1 + |\overline{R}| + \cdots + |\overline{R}|^{k-1} = (|\overline{R}|^k-1) / (|\overline{R}|-1)$ nonassociate atoms where $k = \text{dim}_{\overline{R}} \; M / M^2$.  But $1 + |\overline{R}| + |\overline{R}|^2 \geq 2 |\overline{R}|$, so we must have $k=2$ in which case $R$ has $|\overline{R}|+1$ nonassociate atoms.
\end{proof}

Let $R = GF(2)[[X^2,X^3]] = GF(2) + GF(2)[[X]]X^2$, so $\overline{R} = GF(2)$ and hence $|\overline{R}|=2$ where $\overline{R} = R/M$, $M$ the maximal ideal of $R$.  By Example \ref{CKex}, $R$ has exactly $4 = 2 |\overline{R}|$ nonassociate atoms, but $M^2$ is not universal.

We can now improve on Theorem \ref{QLatom} (6) which stated that if $(R,M)$ is a local CK domain with an atom in $M^2$, then the number of nonassociate atoms of $R$ is at least $|\overline{R}| (|\overline{R}|^k-1) / (|\overline{R}|-1)+1 = (|\overline{R}|^{k+1}-1) / (|\overline{R}|-1)$ where $\overline{R} = R/M$ and $k = \text{dim}_{|\overline{R}|} M / M^2$.

\begin{cor}
Let $(R,M)$ be a local CK domain with an atom in $M^2$.  Let $\overline{R} = R/M$ and $k = \text{dim}_{\overline{R}} M / M^2$.  Then the number of nonassociate atoms of $R$ is at least $|\overline{R}| (|\overline{R}|^k-1) / (|\overline{R}|-1)+|\overline{R}| = (|\overline{R}|^{k+1}-1) / (|\overline{R}|-1) + |\overline{R}|-1$.  If further $M$ is the maximal ideal of $[M:M]$, there are at least $(|\overline{R}|^{k+1}-1) / (|\overline{R}|-1) + |\overline{R}|$ nonassociate atoms.
\end{cor}

\section{Examples}

This section consists of examples.  We begin by stating the following example from \cite{AMO} showing for each $n \geq 2$ the existence of a local CK domain $(R,M)$ with an atom in $M^n \backslash M^{n+1}$, thus answering a question raised by Cohen and Kaplansky \cite{CK}.  While not noted in \cite{AMO}, we show here that $M^{2n}$ is universal while $M^{2n-1}$ is not.  Thus in a local CK domain $(R,M)$, the least power of $M$ that is universal can be arbitrarily large.

\begin{ex} (\cite[Example 7.3]{AMO}) \label{AndMottEx}
Let $K$ be a finite field and let $n \geq 2$.  Then there is a complete local CK domain $(R,M)$ with $R/M \cong K$ and an atom $f \in M^n \backslash M^{n+1}$.  Moreover, no element of $M^{n+1}$ is an atom.  Here $M^{2n}$ is universal but $M^{2n-1}$ is not universal.  Let $f_n \in K[Y]$ be irreducible of degree $n$.  Let $F$ be a field extension of $K$ with $[F:K]=n+1$ and let $1,y,\ldots,y^{n-1}$ be a $K$-basis for $F$.  For $i$, $1 \leq i \leq n-1$, put $V_i=K \cdot 1 + K \cdot y + \cdots + K \cdot y^i$ and $R = K + V_1 X + \cdots + V_{n-1}X^{n-1}+F[[X]]X^n$.  So $R$ is a local CK domain with maximal ideal $M=V_1 X + \cdots + V_{n-1}X^{n-1}+F[[X]]X^n$.  For $i \geq n$, $M^i = F[[X]]X^i$.  Let $f=f_n(y)X^n$.  Then $f \in M^n \backslash M^{n+1}$ is an atom.  Let $g \in R$ with $\text{ord}(g)=j$.  Then $Rg \supseteq M^n g = F[[X]]X^ng = F[[X]]X^{n+j}=M^{n+j}$.  This shows that there is no atom in $M^{n+1}=F[[X]]X^{n+1}$ (take $g=X$) and hence that $M^{2n}=F[[X]]X^{2n}$ is universal.  However, $M^{2n-1} \not \subseteq Rf$, so $M^{2n-1}$ is not universal.  For suppose $M^{2n-1} \subseteq Rf$, then $F[[X]]X^{2n-1}=M^{2n-1}=M^{n-1}f$.  Let $g \in M^{n-1}f$ with $\text{ord}(g)=2n-1$.  Then the leading coefficient of $g$ is in the $(n-1)$-dimensional $K$-subspace $V_{n-1}f = \{vf_n \; | \; v \in V_{n-1}\} \subsetneq F$, contradicting our assumption that $M^{n-1} f = F[[X]]X^{2n-1}$.
\end{ex}

For the case $n=2$, the ring $R$ has the form $R = K + WX + F[[X]]X^2$ where $K \subsetneq F$ is a field extension and $W$ is a $K$-subspace of $F$.  As we will see this example has exactly $8$ nonassociate atoms with $2$ atoms in $M^2$ where $M$ is the maximal ideal of $R$.

Thus we begin (Example \ref{GenEx}) with a careful study of quasilocal domains of the form $R = K + WX + F[[X]]X^2$ where $K \subseteq F$ is an arbitrary field extension and $W$ is a $K$-subspace of $F$ (possibly $0$).  Such a domain is a BFD and is a CK domain if and only if $K =W = F$ or $F$ is finite.  We completely determine the atoms of $R$ and the atoms of $R$ lying in $M^2$, if any.  We determine the cardinality of the set of nonassociate atoms of $R$ that lie in $M \backslash M^2$ and the cardinality of the set of nonassociate atoms of $R$ that lie in $M^2$.  For $R$ a CK domain, these numbers are given in Example \ref{CKex}.  We then give our example (Example \ref{8atom}) of a local CK domain with $8$ atoms, 2 of which are in $M^2$.  As previously remarked, this contradicts a statement of Cohen and Kaplansky.

We end this section by giving a construction of a family of local CK domains with precisely $3$ nonassociate atoms.

\begin{ex} \label{GenEx}
Let $K \subseteq F$ be a field extension and let $W$ be a $K$-subspace of $F$, possibly $0$.  Let $R = K + WX + F[[X]]X^2$.  So $R$ is a quasilocal domain with maximal ideal $M = WX + F[[X]]X^2$.  For $n \geq 1$, $W^n :=\{w_1 \cdots w_n | w_i \in W\}$ and $KW^n$ denotes the $K$-subspace of $F$ spanned by $W^n$.  Now $\bigcap_{n=1}^{\infty} M^n =0$, so $R$ is a BFD.  We have $R$ is Noetherian if and only if $[F:K] < \infty$, and $R$ is a CK domain if and only if $R$ is a DVR (that is, $F=W=K$) or $F$ is finite.  Furthermore, $R$ has residue field $\overline{R}=R/M=K$.  The quotient field of $R$ is $F[[X]][X^{-1}]$, its complete integral closure is $R^c=F[[X]]$, and its integral closure is $R' = L + F[[X]]$ where $L$ is the algebraic closure of $K$ in $F$.

We have that $[M:M]=[W:_F W]+F[[X]]X$ where $K \subseteq [W:_F W] \subseteq F$ and $[W:_F W] = \{x \in F\; | \; xW \subseteq W \}$ is an integral domain.  Additionally, $U(R) = \{\sum_{n=0}^{\infty} a_n X^n \in R \; | \; a_0 \neq 0 \}$ and 

\begin{align*}
U([M:M]) &= \{\sum_{n=0}^{\infty} a_n X^n \in [M:M] \; | \; a_0 \in U([W:_F W] \} \\
&= \{\sum_{n=0}^{\infty} a_n X^n \in F[[X]] \; | \; a_0 \in U([W:_F W]) \}.
\end{align*}

We have the s.e.s. $$0 \to U(R^c) / U(R) \to G(R) \to G(R^c) \to 0$$ which splits since $G(R^c) \cong \Z$, so $G(R) \cong \Z \oplus U(R^c) / U(R)$.  Now $U(R^c) /U(R) \cong F^* / K^* \oplus F/W$, for the map $$\sum_{n=0}^{\infty} a_n X^n \to (a_0 K^*, a_0^{-1} a_1 + W),$$ $\sum_{n=0}^{\infty} a_n X^n \in F[[X]]$, $a_0 \neq 0$, is a homomorphism with kernel $U(R)$.  With this identification $V:=U([M:M]) /  U(R) \cong U([W:_F W]) / K^* \oplus F/W$.

Let $f = a_n X^n + a_{n+1}X^{n+1}+ \cdots \in R$ where $n \geq 0$ and $a_n \neq 0$.  Then $f \sim a_n X^n + a_{n+1}X^{n+1}$.  Suppose that $a X^n + bX^{n+1} \sim c X^n + d x^{n+1}$ in $R$, $a,c, \neq 0$.  Then $a c^{-1} \in K^*$, so after multiplying $cX^n + dX^{n+1}$ by $ac^{-1}$, it suffices to determine when $aX^n + bX^{n+1} \sim aX^n + d X^{n+1}$.  But this holds if and only if $(aX^n + bX^{n+1}) (aX^n + dX^{n+1})^{-1} \in U(R)$ $\iff$ $\frac{b}{a} - \frac{d}{a} \in W$ $\iff$ $b + aW = d + aW$.  Here $aW$ is a $K$-subspace of $F$.

Let $\{a_{\alpha}\}_{\alpha \in \Lambda}$ be a complete set of representatives of $F^* / K^*$.  Equivalently, we have that $\{a_{\alpha}\}_{\alpha \in \Lambda}$ is a complete set of representatives of the one-dimensional $K$-subspaces of $F$ ($a_{\alpha}K^* \leftrightarrow K a_{\alpha}$).  Let $\{b_{\beta}\}_{\beta \in \Gamma}$ be a complete set of representatives of $F/W$.  For $a \in F^*$, $\{ab_{\beta}\}_{\beta \in \Gamma}$ is a complete set of representatives of $F/aW$.  We have $|\Lambda| = |F^* / K^*|$ and $|\Gamma| = |F/W| = |F/aW|$.  For $W \neq 0$, we let $\{a_{\alpha}\}_{\alpha \in \Omega} \subseteq \{a_{\alpha}\}_{\alpha \in \Lambda}$ be a complete set of representatives of the one-dimensional $K$-subspaces of $W$, or equivalently, of the cosets $\{w K^* \; | \; w \in W \backslash \{0\}\}$.

Suppose $n \geq 2$.  Now $a_{\alpha}X^n + bX^{n+1} \sim a_{\beta}X^n + dX^{n+1} \implies \alpha = \beta$ and $a_{\alpha}X^n + bX^{n+1} \sim a_{\alpha}X^n + dX^{n+1} \iff b + a_{\alpha}W = d + a_{\alpha}W$.  Thus we have that the set $\{a_{\alpha}X^n + a_{\alpha}b_{\beta}X^{n+1}\}_{(\alpha,\beta) \in \Lambda \times \Gamma}$ is a complete set of representatives for the equivalence classes of associate elements of the form $aX^n + bX^{n+1}$, $0 \neq a \in F$.  Next suppose that $n=1$ and $W \neq 0$.  Then $\{a_{\alpha}X + a_{\alpha}b_{\beta}X^2\}_{(\alpha,\beta)\in \Omega \times \Gamma}$ is a complete set of representatives for the equivalence classes of associate elements of the form $aX + bX^2$ where $0 \neq a \in W$.

We next determine the atoms of $R$.  Let $f = a_n X^n + a_{n+1}X^{n+1} + \cdots \in R$ where $a_n \neq 0$, so $\text{ord}(f)=n$.  Now $f$ is a unit $\iff$ $\text{ord}(f)=0$.  If $n = \text{ord}(f)=1$, $f$ is an atom.  If $\text{ord}(f) \geq 4$, $f$ is never an atom.  Suppose $W \neq 0$.  Let $0 \neq w \in W$.  Let $\text{ord}(f)=3$.  Then $f = w X (w^{-1}a_3 X^2+w^{-1}a_4X^3 + \cdots )$ and hence $f$ is not an atom.  For $W=0$ and $\text{ord}(f)=3$, $f$ is always an atom.  We next determine when $f$ is an atom for $\text{ord}(f)=2$.  Now $f = a_2X^2+a_3X^3 + \cdots$ is an atom $\iff$ $a_2 X^2 + a_3 X^3$ is an atom.  Since $f$ is not an atom $\iff$ $a \in W^2$, we have that $a_2 X^2 + a_3 X^3$ is an atom $\iff$ $a_2 \in F \backslash W^2$.

Case $W = 0$.  So $R = K + F[[X]]X^2$, $M^n = F[[X]]X^{2n}$, $[M:M]=F[[X]]$, $G(R) \cong \Z \oplus F^* /K^* \oplus F$ and $V \cong F^* / K^* \oplus F$, under this identification.  Also, $\{a_{\alpha}X^2 + a_{\alpha}bX^3 | \alpha \in \Lambda, b \in F\}$ or $\{a_{\alpha}X^2 + bX^3 \; | \; \alpha \in \Lambda, b \in F\}$ is a complete set of nonassociate atoms of $R$ of order $2$.  Hence this set of nonassociate atoms has cardinality $|F^* /K^*||F|$.  Likewise $\{a_{\alpha}X^3 + a_{\alpha}bX^4 \; | \; a \in \Lambda, b \in F\}$ or $\{a_{\alpha} X^3 + bX^4 \; | \; \alpha \in \Lambda, b \in F\}$ is a complete set of nonassociate atoms of order $3$ and has cardinality $|F^*/K^*||F|$.  So the set of nonassociate atoms of $R$ has cardinality $2|F^*/K^*||F|$.  There are no atoms in $M^2$.  Here $M^2$ is not universal, but $M^3$ is universal.

Case $W \neq 0$.

\begin{itemize}
\item[a)]  $W=F$, so $R = K + F[[X]]X$.  Here $M^2$ is universal and $\{a_{\alpha}X\}_{\alpha \in \Lambda}$ is a complete set of nonassociate atoms of $R$.  All atoms are in $M \backslash M^2$.  The cardinality of the set of atoms is $|F^*/K^*|$.  We can identify $V$ with $F^*/K^*$ where $G(R) \cong \Z \oplus F^*/K^*$.

% Huh?  Don't get the cardinality thing that was highlighted
\item[b)]  $W \neq F$, so $0 \subsetneq W \subsetneq F$.  Here atoms have order $1$ or order $2$, so there are no atoms in $M^3$.  We see that $\{a_{\alpha}X + a_{\alpha}b_{\beta}X^2\}_{(\alpha,\beta) \in \Omega \times \Gamma}$ is a complete set of nonassociate atoms of order $1$, of course, all lying in $M \backslash M^2$.  So this gives $|\Omega| |F/W|$ nonassociate atoms.  Now $aX^2 + bX^3$ is an atom $\iff$ $a \notin W^2$.  Thus $\{a_{\alpha}X^2 + a_{\alpha}b_{\beta}X^3 \; | \; a_{\alpha} \in F^* \backslash W^2, \beta \in \Gamma\}$ is a complete set of nonassociate atoms of order $2$.  The cardinality of this set is $|\{a_{\alpha} \; | \; \alpha \in \Lambda, a_{\alpha} \in F^* \backslash W^2\}| |F/W|$.  Note that $|\{a_{\alpha} \; | \; \alpha \in \Lambda, a_{\alpha} \in F^* \backslash W^2\}| = |\{aK^* \; | \; a \in F^* \backslash W^2\}|=|\{Ka \; | \; a \in F^* \backslash W^2\}|$.  So the cardinality of the set of nonassociate atoms of $R$ is $|\Omega| |F/W| + |\{a_{\alpha} \; | \; \alpha \in \Lambda, a_{\alpha} \in F^* \backslash W^2\}| |F/W|$.  Note that the atom (of order 2) $a_{\alpha} X^2 + a_{\alpha} b_{\beta}X^3$ is in $M^2 = KW^2 X^2 + F[[X]]X^3 \iff a_{\alpha} \in KW^2 \backslash W^2$.  Thus there is an atom in $M^2$ $\iff$ $W^2 \subsetneq KW^2$.  The set $\{a_{\alpha} X^2 + a_{\alpha}b_{\beta}X^3 \; | \; a_{\alpha} \in KW^2 \backslash W^2, \beta \in \Gamma\}$ is a complete set of nonassociate atoms of $R$ in $M^2$.  The cardinality of this set is $|\{a_{\alpha} \; | \; \alpha \in \Lambda, a_{\alpha} \in KW^2 \backslash W^2\}| |F/W|$.

Here $M^2$ is never universal, but $M^4$ is always universal.  Let $f$ be an atom of $R$.  If $\text{ord}(f)=1$, $Rf \supseteq M^3 = KW^3 + F[[X]]X^4$.  So suppose $\text{ord}(f)=2$, so $f \sim aX^2 + bX^3$ where $a \in F^* \backslash W^2$ and $b \in F$.  Now $Rf \supseteq M^3 \implies F = Wa$.  Conversely, if $a \in F^* \backslash W^2$ with $F = aW$, then $R(aX^2 + bX^3) \supseteq M^3$ for any $b \in F$.  Thus $M^3$ is universal if and only if $F=aW$ for each $a \in F^* \backslash W^2$.  However $F=aW$ $\iff$ $F = a^{-1}F=W$.  But we are assuming that $W \neq F$.  Thus $M^3$ is universal if and only if $F=W^2$.
\end{itemize}
\end{ex}

% Anything missing?

We summarize the results for Example \ref{GenEx} for the case where $R$ is a CK domain.

\begin{ex} \label{CKex}
Let $R = K+WX + F[[X]]X^2$ where $K \subseteq F$ is an extension of finite fields and $W$ is a $K$-subspace of $F$.  Let $M$ be the maximal ideal of $R$.  We have $R^c = R' = F[[X]]$ and $[M:M]=[W:_F W] + F[[X]]X$.  Here $[W:_F W]$ is an intermediate field of $K \subseteq F$.  Let $\{a_{\alpha}\}_{\alpha \in \Lambda}$ be a complete set of representatives of $F^*/K^*$ (or equivalently, of the one-dimensional $K$-subspaces of $F$).  Let $\Omega = \{\alpha \in \Lambda \; | \; a_{\alpha} \in W\}$.  So $|\Lambda| = (|F|-1)/(|K|-1)$ and $|\Omega| = (|W|-1)/(|K|-1)$.  Let $\{b_{\beta}\}_{\beta \in \Gamma}$ be a complete set of representatives of $F/W$, so $|\Gamma|=|F|/|W|$.  Here $G(R) \cong \Z \oplus F^* / K^* \oplus F/W$ and we can identify $V$ with $[W:_F W]^* / K^* \oplus F/W$.  So $|V| = ((|[W:_F W]|-1) / (|K|-1) ) (|F| / |W|)$.

\begin{itemize}
\item[(1)]  $R$ is a DVR $\iff$ $K=W=F$.  (Here we don't need that $F$ is finite.)

\item[(2)]  $M^2$ is universal $\iff$ $W=F$.  In this case $\{a_{\alpha}X\}_{\alpha \in \Lambda}$ is a complete set of nonassociate elements of $R$.  So the number of nonassociate atoms is $(|F|-1)/(|K|-1)$.  There is one $V$-class of nonassociate atoms.

\item[(3)]  Suppose $W=0$, so $R=K+F[[X]]X^2$.  Then $\{a_{\alpha}X^2 + bX^3 \; | \; \alpha \in \Lambda, b \in F\}$ (resp., $\{a_{\alpha}X^3 + bX^4 \; | \; \alpha \in \Lambda, b \in F\}$) is a complete set of nonassociate atoms of $R$ of order $2$ (resp., order $3$).  So the number of nonassociate atoms of $R$ is $2 ((|F|-1)/(|K|-1)) |F|$.  We have $M^3$ is universal, but $M^2$ is not.  There are no atoms in $M^2$.  Here $[M:M]=F[[X]]$ and $V \cong F^* /K^* \oplus F$.  So $|V| = ((|F|-1)/(|K|-1))|F|$ and there are two $V$-classes of nonassociate atoms.

\item[(4)]  Suppose that $0 \subsetneq W \subsetneq F$.  Then $\{a_{\alpha}X + a_{\alpha}b_{\beta}X^2 \; | \; \alpha \in \Omega, \beta \in \Gamma \}$ is a complete set of nonassociate atoms of $R$ of order $1$.  Their cardinality is $(|W|-1)/(|K|-1)(|F|/|W|)$.  And $\{a_{\alpha}X^2 + a_{\alpha} b_{\beta}X^3 \; | \; \alpha \in \Lambda, a_{\alpha} \in F^* \backslash W^2, \beta \in \Gamma \}$ is a complete set of nonassociate atoms of $R$ of order $2$.  Their number is $m (|F|/|W|)$ where $m = |\{a_{\alpha} \; | \; \alpha \in \Lambda, a_{\alpha} \in F^* \backslash W^2\}| = |\{aK^* \; | \; a \in F^* \backslash W^2\}| = |\{Ka \; | \; a \in F^* \backslash W^2 \}|$.  So $R$ has $((|W|-1)/(|K|-1)+m)(|F|/|W|)$ nonassociate atoms.  There is an atom in $M^2$ (and hence in $M^2 \backslash M^3$) if and only if $W^2 \subsetneq KW^2$.  In this case $\{a_{\alpha}X^2 + a_{\alpha} b_{\beta}X^3 \; | \; \alpha \in \Lambda, a_{\alpha} \in KW^2 \backslash W^2, \beta \in \Gamma \}$ is a complete set of nonassociate atoms in $M^2$.  The cardinality of this set is $m' (|F|/|W|)$ where $m' = |\{a_{\alpha} \; | \; \alpha \in \Lambda, a_{\alpha} \in KW^2 \backslash W^2 \}|=|\{aK^* \; | \; a \in KW^2 \backslash W^2 \}|=|\{Ka \: | \; a \in KW^2 \backslash W^2 \}|$.  We have $M^3$ is universal if and only if $F=W^2$.  Otherwise $M^4$ is universal.  Here $K \subseteq [W:_F W] \subseteq F$ is an intermediate field and $|V| = (|[W:_F W]|-1)/(|K|-1)(|F|/|W|)$.
\end{itemize}
\end{ex}

We specialize further to the case where $W$ is an intermediate field $L$, $K \subseteq L \subseteq F$, with $F$ still finite.

\begin{ex} \label{Wint}
Let $L$ be an intermediate field of the field extension $K \subseteq F$ where $F$ is finite and let $R = K + LX + F[[X]]X^2$.  So $M = LX + F[[X]]X^2$ is the maximal ideal of $R$ and $[M:M]=L + F[[X]]X$.  Here $M^4$ is universal, but $M^3$ is universal $\iff$ $M^2$ is universal $\iff$ $L=F$.  There are no atoms in $M^2$.  We see that $R$ has $((|F|-1) / (|K|-1)) |F| /|L|$ nonassociate atoms, $|V| = (|L|-1) / (|K|-1)$, so there are $((|F|-1) / (|L|-1)) |F| / |L|$ $V$-classes of nonassociate atoms.
\end{ex}

We now give our example of a local CK domain $(R,M)$ with $8$ atoms having an atom in $M^2$.

\begin{ex} \label{8atom}
Let $\{1,y,y^2\}$ be a GF(2)-basis for GF($2^3$) where $y$ is a zero of the irreducible polynomial $Y^3+Y+1 \in GF(2)[Y]$.  Let $W$ be the subspace of GF($2^3$) spanned by $1$ and $y$ and $R = GF(2)+WX + GF(2^3)[[X]]X^2$.  (This is \cite[Example 7.3]{AMO} for the case $n=2$.)  Here $W^2 = \{0,1,y,1+y,y^2,1+y^2,y+y^2\}$, so $GF(2)W^2 = GF(2^3)$.  Now $\{1,y,1+y,y^2,1+y^2,y+y^2,1+y+y^2 \}$ is a complete set of representatives of the one-dimensional $GF(2)$-subspaces of $GF(2^3)$ and those lying in $W$ are $\{1,y,1+y\}$.  We take $0,1+y+y^2$ as a complete set of representatives of $GF(2^3) / W$.  So the order $1$ atoms are $X, X+(1+y+y^2)X^2, yX, yX+y(1+y+y^2)X^2 = yX+(1+y^2)X^2, (1+y)X, \text{ and } (1+y)X + (1+y)(1+y+y^2)X^2 = (1+y)X + yX^2$.  These $6$ atoms are in $M \backslash M^2$.  Now $GF(2) W^2 \backslash W^2 = F \backslash W^2 = \{1+y+y^2\}$.  So the two remaining atoms of $R$ are $(1+y+y^2)X^2$ and $(1+y+y^2)X^2 + (1+y+y^2)^2X^3 = (1+y+y^2)X^2 + (1+y)X^3$, both of which lie in $M^2 \backslash M^3$.   Here $[W:_{GF(2)} W] = GF(2)$ since $1 \in W$, so $[M:M]=GF(2) + GF(2^3)[[X]]X$.  So we can take $1$ and $1+(1+y+y^2)X$ as representations of $V = U([M:M]) /U(R)$.  Here the $V$ classes are $\{X, X(1+(1+y+y^2)X)\}, \{yX, yX(1+(1+y+y^2)X)\}, \{(1+y)X, (1+y)X(1+(1+y+y^2)X)\}, \text{ and } \{(1+y+y^2)X, (1+y+y^2)X(1+(1+y+y^2)X)\}$.  Finally, since $GF(2^3) \neq W^2$, $M^3$ is not universal, but $M^4$ is universal.
\end{ex}

More generally we have the following example whose proof is similar to that of Example \ref{8atom}.  It is interesting to note that except for the case of $p=2$ or $3$, and $n=1$, the local CK domain $(R,M)$ in this example has more nonassociate atoms in $M^2$ than in $M \backslash M^2$.

\begin{ex} \label{ExAtomsInM2}
Let $p$ be a prime number and $n$ a natural number.  Let $\{1,y,y^2\}$ be a GF($p^n$)-basis for GF($p^{3n}$) where $y$ is a zero of an irreducible cubic $f(Y) \in \text{GF}(p^n)[Y]$.  Let $W$ be the GF($p^n$)-subspace $\text{GF}(p^n)\cdot 1 + \text{GF}(p^n) \cdot y$ and $R$ be the ring $R = \text{GF}(p^n) + WX + \text{GF}(p^{3n})[[X]]X^2$.  Then $(R,M)$ is a local CK domain with $M^4$ universal (but $M^3$ is not universal).  Here $|W^2| = \frac{p^{3n}}{2}+p^{2n}-\frac{p^n}{2}$ so $|\text{GF}(p^3) \backslash W^2 | = p^{3n} - \left( \frac{p^{3n}}{2}+p^{2n}-\frac{p^n}{2} \right) = \frac{p^n (p^n-1)^2}{2}$.  So $|\{\text{GF}(p^{3n})^*a \; | \; a \in \text{GF}(p^{3n}) \backslash W^2 \} | = \frac{p^n(p^n-1)}{2}$.  So from Example \ref{GenEx} we have

\begin{align*}
|V| = p^n, \\
p^n(p^n+1) & \text{ nonassociate atoms in } M \backslash M^2, \\
\frac{p^{2n}(p^n-1)}{2} & \text{ nonassociate atoms in } M^2 \backslash M^3,\text{ and }\\
\frac{p^n(p^{2n}+p^n+2)}{2} & \text{ total nonassociate atoms.}
\end{align*}

For small values of $p^n$ we have:

\begin{center}
  \begin{tabular}{ | c | c | c | c |}
    \hline
    $p^n$ & Atoms in $M \backslash M^2$ & Atoms in $M^2$ & Atoms \\ \hline
    2 & 6 & 2 & 8 \\ \hline
    3 & 12 & 9 & 21 \\ \hline
    4 & 20 & 24 & 44 \\ \hline
    5 & 30 & 50 & 80 \\ \hline
    7 & 56 & 147 & 203 \\ \hline
    8 & 72 & 224 & 296 \\ \hline
    9 & 90 & 324 & 414 \\ \hline
  \end{tabular}
\end{center}
\end{ex}

After pointing out an error in \cite{CK}, it is time to point out a mathematical error and typographical error in \cite{AMO} and to note a partial correction.  In \cite[Theorem 7.1]{AMO} it stated for $R = K+V_1X + \cdots + V_{n-1} X^{n-1}+F[[X]]X^n$ where $K \subseteq F$ is a field extension and $V_1,\ldots,V_{n-1}$ are $K$-subspaces of $F$ with $V_i V_j \subseteq V_{i+j}$ for $i+j \leq n-1$, that $G(R) \cong F^* / K^* \oplus F/V_1 \oplus \cdots \oplus F / V_{n-1}$.  In the proof it is alleged that the map $\pi: U(F[[X]]) \to F^* / K^* \oplus F/V_1 \oplus \cdots \oplus F / V_{n-1}$ given by $\pi (a_0 (1+a_1 X + a_2X^2 + \cdots)) = (a_0K^*,a_1+V_1,\ldots,a_{n-1}+V_{n-1})$ is a homomorphism.  However, this is only the case for $n=1$.  Thus we only have the isomorphism $G(K+WX + F[[X]]X^2) \cong \Z \oplus F^* / K^* \oplus F/W$ which is given in Example \ref{GenEx}.  Corollary 7.2 of \cite{AMO} concerns the special case where $R = K+F[[X]]X^n$.  The assertion that $G(R) \cong F^* / K^* \oplus F^{n-1}$ is only valid for $n=1,2$.  Also, there is a typographical error in giving the number of nonassociate atoms of $R$ as $n |F^*/K^*| |F|^n$ where obviously the correct number is $n |F^* /K^*| |F|^{n-1}$.  The proof given is correct once the typographical error is corrected.  However, as we will use this example, we state the correct result with a self-contained proof.

\begin{ex} \label{Corrected}
Let $K \subseteq F$ be a field extension, $n \geq 1$, and $R = K+F[[X]]X^n$.  Then $R$ is a BFD but is local (resp., a CK domain) if and only if $[F:K] < \infty$ (resp., $K=F$ and $n=1$ or $F$ is finite).  Let $\{a_{\alpha}\}_{\alpha \in \Lambda}$ be a complete set of representatives of $F^*/K^*$.  Then $\{a_{\alpha}X^i + b_1 X^{i+1} + \cdots + b_{n-1}X^{i+n-1} \; | \; \alpha \in \Lambda, b_i \in F, n \leq i \leq 2n-1\}$ is a complete set of nonassociate atoms of $R$.  Thus $R$ has $n |F^*/K^*| |F|^{n-1}$ nonassociate atoms.  We have $M^3$ is universal, but $M^2$ is not universal unless $n=1$.  There are no atoms in $M^2$ where $M$ is the maximal ideal of $R$.  It follows from \cite[Corollary 5.6]{AMO} that atoms have the form $u(X)X^i$ where $u(X) \in U(F[[X]])$ and $n \leq i \leq 2n-1$ and the number of nonassociate atoms is $n |F^*/K^*| |U(F[[X]]/U(R)|^{n-1}$.  But we prove this directly.  It is easy to see that atoms of $R$ have the form $a_i X^i + a_{i+1}X^{i+1}+ \cdots$ where $a_i \neq 0$ and $n \leq i \leq 2n-1$.  But $a_i X^i + a_{i+1}X^{i+1} + \cdots = u(X) X^i$ where $u(X)=a_i + a_{i+1}X + \cdots \in U(F[[X]])$.  Note that $u(X)X^i \sim v(X)X^j \iff i=j$ and $u(X)U(R) = v(X)U(R)$.  But $|U(F[[X]])/U(R)| = |F^*/K^*||F|^{n-1}$ as seen by the bijection given by $\left(\sum_{m=0}^{\infty} a_m X^m \right) U(R) \leftrightarrow (a_0 K^*,a_0^{-1}a_1,\ldots,a_0^{-1}a_{n-1})$ between $U(F[[X]])/U(R)$ and $F^*/K^* \oplus F^{n-1}$.  (This map is an isomorphism only for $n \leq 2$.)
\end{ex}

Let $(R,M)$ be  a local CK domain with $\overline{R}=R/M$ and $k = \text{dim}_{\overline{R}} M/M^2$.  Then by Theorem \ref{QLatom} (6) $R$ has at least $(|\overline{R}|^k-1) / (|\overline{R}|-1)$ nonassociate atoms.  Thus the number of nonassociate atoms gives an upper bound for $|\overline{R}|$ and $k$.  However, for a given $|\overline{R}|$ and $k$, the number of nonassociate atoms can be arbitrarily large.  We illustrate this for the case $\overline{R} \cong GF(2)$ and $k=2$.

\begin{ex}
(An example of a local CK domain $(R,M)$ with $\overline{R} = R/M \cong GF(2)$ and $\text{dim}_{\overline{R}} M/M^2 = 2$ having more than $2^{n-1}+1$ nonassociate atoms for $n \geq 2$)  Let $f$ be a principal prime of $GF(2) [[X,Y]]$ with $\text{ord}(f) = n \geq 2$.  Take $R = GF(2)[[X,Y]] / (f)$ so $R$ is a complete local CK domain with $\overline{R} \cong GF(2)$ and having $\text{dim}_{\overline{R}} M/M^2 = 2$.  Suppose that $R$ has $m$ nonassociate atoms $a_1,\ldots,a_m$.  Taking $f_i \in GF(2) [[X,Y]]$ with $\overline{f_i} = a_i$, we have $M = (a_1)\cup \cdots \cup (a_m)$ and hence $(X,Y) = (f_1,f) \cup \cdots \cup (f_m,f) = (f_1) + (X,Y)^n \cup \cdots \cup (f_m) + (X,Y)^n$.  We claim that $m \geq 2^{n-1}+1$.  Put $S = GF(2)[[X,Y]] / (X,Y)^n$ and $N = (X,Y) / (X,Y)^n$.  Then $N$ is a union of $m$ principal ideals.  But $|S|=2^{\frac{n(n+1)}{2}}$, $|N|=2^{\frac{(n-1)(n+2)}{2}}$, and if $Sa$ is a proper principal ideal, $$|Sa| = |S| / |\text{ann}(a)| \leq |S| / |N^{n-1}| \leq 2^{(n(n+1)/2} / 2^n = 2^{(n-1)n/2}.$$  Since two principal ideals have nonempty intersection, $$m \geq |N| / |Sa| + 1 = 2^{(n-1)(n+2)/2} / 2^{(n-1)n/2}+1 = 2^{n-1}+1.$$  So $R$ has at least $2^{n-1}+1$ nonassociate atoms.
\end{ex}

The following well known diagram appears in \cite{AAZ} where examples are given to show that none of the implications can be reversed.

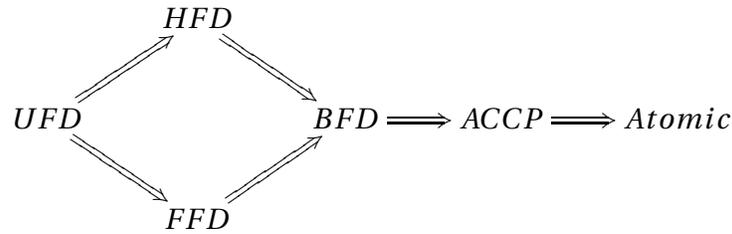
\[
\xymatrix{
&HFD \ar@{=>}[rd] \\
UFD \ar@{=>}[ru] \ar@{=>}[rd] & &BFD \ar@{=>}[r] &ACCP \ar@{=>}[r] &Atomic\\
&FFD \ar@{=>}[ru]}
\] 

Let us extend this diagram for the case of a quasilocal domain $(R,M)$.

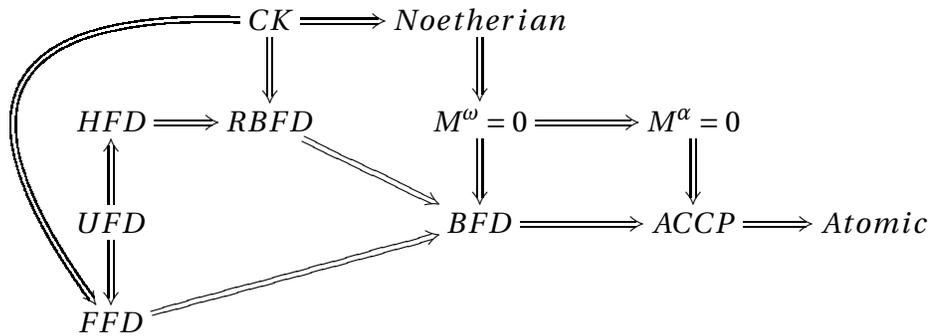
\[
\xymatrix{
& CK \ar@{=>}[r] \ar@{=>}[d] \ar@{=>}@/_6pc/[lddd] & Noetherian \ar@{=>}[d] \\
HFD \ar@{=>}[r] & RBFD \ar@{=>}[rd] & M^{\omega}=0 \ar@{=>}[r] \ar@{=>}[d] & M^{\alpha}=0 \ar@{=>}[d] \\
UFD \ar@{=>}[u] \ar@{=>}[d] & & BFD \ar@{=>}[r] &ACCP \ar@{=>}[r] & Atomic \\
FFD \ar@{=>}[rru]}
\] 

Here $R$ is a RBFD if $R$ has finite elasticity $\rho(R)$.  A one-dimensional local domain is a RBFD if and only if $R$ is analytically irreducible \cite[Theorem 2.12]{AA}.  Also, a one-dimensional local domain is an FFD if and only if $R$ is a DVR or $R/M$ is finite \cite[Corollary 6]{AMU}.  Thus a one-dimensional local domain is a CK domain if and only if it is an FFD and RBFD.  Except for the implication $M^{\alpha}=0$ $\implies$ ACCP, we give examples of quasilocal domains showing that the implications cannot be reversed.  In fact, in all cases except $M^{\omega}=0$ $\implies$ BFD we give one-dimensional quasilocal examples.  (1) BFD $\notimplies M^{\omega} = 0$:  Localizing \cite[Example 5.7]{HLV} gives an example of a quasilocal Krull domain $(R,M)$ with $\bigcap_{n=1}^{\infty} M^n \neq 0$.  But a Krull domain is a BFD.  (2) atomic $\notimplies$ ACCP: Gram's example \cite{G} of an atomic domain not satisfying ACCP is a one-dimensional quasilocal domain.  (3) ACCP $\notimplies$ BFD, $M^{\alpha} =0 \notimplies M^{\omega}=0$: take \cite[Example 2.1]{AAZ} $K[X;T]$, $K$ a field and $T$ the additive submonoid of $\Q^+$ generated by $\{1/p | p \text{ is prime }\}$, localized at $N = \{f \in K[X;T] | f \text{ has nonzero constant term}\}$.  (4) $M^{\omega} =0 \notimplies$ Noetherian: $\Q+\C[[X]]X$. (5) HFD $\notimplies$ UFD, BFD $\notimplies$ FFD, RBFD $\notimplies$ CK, Noetherian $\notimplies$ CK: $\R+\C[[X]]$.  (6) FFD $\notimplies$ UFD, RBFD $\notimplies$ HFD, FFD $\notimplies$ CK: $K[[X^2,X^3]]$, $K$ a field (with $K$ infinite for FFD $\notimplies$ CK). (7) BFD $\notimplies$ RBFD: any one-dimensional local domain that is not analytically irreducible.

We end by investigating local CK domains with exactly three nonassociate atoms.  As we know, in this case $M^2$ is universal.  In the complete equi-characteristic case it is easy to completely characterize such integral domains.  We begin with the following more general result.

\begin{thm}
Let $(R,M)$ be a complete local domain with residue field $\overline{R}$.  Suppose that $M^2$ is universal and char $R =$ char $\overline{R}$.  So the integral closure $R'$ is a complete DVR and hence $R' \cong F[[X]]$ where $F$ is a subfield of $R'$ that maps isomorphically onto the residue field of $R'$.  Then $R \cong K + F[[X]]X$ where $K = R \cap F$ is a subfield of $R$ isomorphic to $\overline{R}$.  Also, $R$ is a CK domain if and only if $K=F$ (and hence $R \cong F[[X]]$ is a DVR) or $\overline{R}$ is finite.
\end{thm}

\begin{proof}
Since $M^2$ is universal, $R'=[M:M]$ is a DVR with maximal ideal $M$ by Theorem \ref{QLdoms} (7).  So $R'$ is a complete DVR with char $R' =$ char $R'/M$.  So $R' \cong F[[X]]$ where $F$ is a subfield of $R'$ that maps isomorphically onto $R'/M$.  Now $R'$ has maximal ideal $M=F[[X]]X$.  Let $K=R \cap F$, so $F[[X]]X \subseteq R$ gives $R = K+F[[X]]X$.  Since $K \subseteq F$ is integral, $K$ is a field and clearly $K$ is isomorphic to $\overline{R}$.

Certainly if $K=F$ (and hence $R$ is a DVR) or $\overline{R}$ is finite, $R$ is a CK domain.  Conversely, suppose that $R$ is a CK domain.  If $\overline{R}$ is infinite, $R$ must be a DVR and hence $K=F$.
\end{proof}

\begin{cor}
Let $(R,M)$ be a complete local CK domain with exactly three nonassociate atoms with char $R= $ char $\overline{R}$.  Then $R \cong GF(2) + GF(2^2)[[X]]X$.
\end{cor}

\begin{proof}
Now $R = GF(2)+GF(2^2)[[X]]X$ is a complete local CK domain with char $R= $ char $\overline{R}$ having exactly $|GF(2^2)^*/GF(2)^*|=3$ nonassociate atoms.  Conversely, suppose that $(R,M)$ is a complete local CK domain with exactly 3 non-associate atoms.  Since $3$ is prime, $M^2$ is universal.  Since $R$ is not a DVR, $\overline{R}$ must be finite.  So $R \cong GF(p^n) + GF(p^{nk})[[X]]X$ for some prime $p$ (= char $\overline{R}$) and $k \geq 2$.  Now $R$ has $3 = |GF(p^{nk})^* / GF(p^n)^*|$ nonassociate atoms.  So we must have $p=2$ and $k=2$.  Thus $R \cong GF(2) + GF(2^2)[[X]]X$.
\end{proof}

There is another way to realize $R = GF(2)+GF(2^2)[[X]]X$.  Here $GF(2^2) = \{0,1,\alpha,1+\alpha\}$ where $\alpha^2 = \alpha+1$.  We claim that $$R \cong GF(2)[[X,Y]] / (X^2+XY+Y^2).$$  Here is a sketch.  The map $\phi: GF(2)[[X,Y]] \to R$ given by $\phi(X)=X$ and $\phi(Y)=\alpha X$ is an epimorphism.  Since $R$ is an integral domain, $\text{ker}(\phi)$ is a prime ideal of $GF(2)[[X,Y]]$.  Now $\phi(X^2+XY+Y^2) = X^2 + \alpha X^2 + \alpha^2 X^2 = (1+\alpha+\alpha^2)X^2 = 0$, so $(X^2+XY+Y^2) \subseteq \text{ker}(\phi)$.  But since $(X^2+XY+Y^2)$ is a prime ideal, we have $\text{ker}(\phi) = (X^2+XY+Y^2)$.  Here we have realized a local CK domain with precisely $3$ nonassociate atoms as a homomorphic image of a two-dimensional regular local ring.  We next generalize this result.

\begin{thm} \label{Image}
Let $(D,M)$ be a two-dimensional regular local domain with $M = (x_1,x_2)$ and let $f$ be a principal prime of $D$.  Then $D/(f)$ is an integral domain with precisely three nonassociate atoms if and only if $D/M \cong \text{GF}(2)$ and $f = u_1x_1^2 + u_2 x_1 x_2 + u_3 x_2^2$ where $u_1,u_2,u_3$ are units.
\end{thm}

\begin{proof}
Note that if $f \in M \backslash M^2$, $\overline{D}:=D/(f)$ is a DVR, and if $\overline{D}$ has exactly three nonassociate atoms, then $D/M \cong \text{GF}(2)$ by Theorem \ref{QLatom} (2).  Now $\overline{D}$ has precisely three nonassociate atoms $\iff$ $\overline{M} = (\overline{x_1}) \cup (\overline{x_2}) \cup (\overline{x_1}+\overline{x_2})$ $\iff$ $M = (x_1,f) \cup (x_2,f) \cup (x_1+x_2,f)$.  Let $f = a_1 x_1^2 + a_2 x_1 x_2 + a_3 x_2^2$.  Then $(x_1,f) = (x_1,a_3x_2^2)$, $(x_2,f) = (x_2,a_1x_1^2)$, and $(x_1+x_2,f) = (x_1+x_2,(a_3-a_2-a_1)x_1x_2)$.

$(\impliedby)$  Suppose that $a_1, a_2$, and $a_3$ are units.  Then $D/M \cong \text{GF}(2)$ gives that $a_3-a_2-a_1$ is a unit.  So $(x_1,f) \cup (x_2,f) \cup (x_1+x_2,f) = (x_1) + M^2 \cup (x_2)+M^2 \cup (x_1+x_2)+M^2 = M$ since $(x_1)+M^2$, $(x_2)+M^2$, and $(x_1+x_2)+M^2$ are the one-dimensional $D/M$-subspaces of $M / M^2$.  $(\implies)$  Suppose that $$M=(x_1,f) \cup (x_2,f) \cup (x_1+x_2,f).$$  If $a_1 \in M$, $(x_2,f) \subseteq (x_2)+M^3$, so $M = (x_1)+M^2 \cup (x_2)+M^3 \cup (x_1+x_2)+M^2$.  Consider $x_2+x_1^2$.  Now $x_2+x_1^2 \in (x_1)+M^2$ $\implies$ $x_2 \in (x_1)+M^2$ and $x_2+x_1^2 \in (x_1+x_2)+M^2$ $\implies$ $x_2 \in (x_1+x_2)+M^2$, both contradictions.  And $x_2+x_1^2 \in (x_2)+M^3$ $\implies$ $M^2 \subseteq (x_2)+M^3$, a contradiction.  Interchanging $x_1$ and $x_2$ shows that $a_3 \in M$ leads to a contradiction.  So $a_1$ and $a_3$ must be units.  Suppose that $a_2 \in M$.  Then $a_3-a_2-a_1 \in M$.  Thus $M = (x_1,f) \cup (x_2,f) \cup (x_1+x_2,f)$ gives $M = (x_1) + M^2 \cup (x_2)+M^2 \cup (x_1+x_2)+M^3$.  Put $y_1 = -x_1$, $y_2 = x_1+x_2$, so $y_1+y_2 = x_2$ and $(y_1,y_2)=M$.  Now $M = (y_1)+M^2 \cup (y_2)+M^3 \cup (y_1+y_2)+M^2$, a contradiction.
\end{proof}

With regard to the element $f$ in Theorem \ref{Image}, we have the following proposition.

\begin{prop}
Let $(D,M)$ be a two-dimensional regular local domain with $M=(x_1,x_2)$.  Suppose that $D/M \cong \text{GF}(2)$.  Then

\begin{itemize}
\item[(1)] $f = u_1x_1^2 + u_2x_1x_2 + u_3x_2^2$ where $u_1,u_2,u_3$ are units if and only if $f = x_1^2 + x_1x_2 + x_2^2 + g$ for some $g \in M^3$.

\item[(2)]  For units $u_1,u_2,u_3$ of $R$, $f = u_1x_1^2 + u_2x_1x_2 + u_3 x_2^3$ is a nonzero principal prime.
\end{itemize}
\end{prop}

\begin{proof}
(1) $(\implies)$  Since $D/M \cong \text{GF}(2)$, $u_i = 1+m_i$ for some $m_i \in M$.  Then $f=x_1^2 + x_1x_2+x_2^2 + g$ where $g = m_1x_1^2 + m_2x_1x_2 + m_3 x_2^2 \in M^3$.  $(\impliedby)$  Suppose $f = x_1^2 + x_1x_2 + x_2^2 + g$ where $g \in M^3$.  Note that $g = ax_1^2 + bx_2^2$ for some $a,b \in M$.  Then $f = x_1^2 + x_1x_2 + x_2^2 + g = (1+a)x_1^2 + x_1x_2 + (1+b)x_2^2$ where $a+1$, $1+b$ are units.  (Note that this shows that we can take $u_2=1$.)

(2)  Note that $D$ is a UFD and $x_1,x_2$ are principal primes.  (The simple proof that a two-dimensional regular local ring is a UFD does not require the more general result that a regular local ring is a UFD.)  We first note that $a_1 x_1^2 + a_2 x_1 x_2 + a_3 x_2^2 = 0$ implies $a_1,a_2,a_3 \in M$.  While this follows from analytic independence, we give a simple proof.  Suppose that $a_1x_1^2 + a_2 x_1 x_2 + a_3 x_2^2 = 0$ and say $a_2$ is a unit.  Then $a_2 x_1 x_2 = -a_1 x_1^2 - a_3 x_2^2$ gives $x_1 | a_3 x_2^2$.  So $x_1 | a_3$.  Likewise $x_2 | a_1$.  So dividing by $x_1 x_2$ gives $a_2 \in M$, a contradiction.  Similar proofs show that $a_1,a_3 \in M$.  Let $f = u_1 x_1^2 + u_2 x_1 x_2 + u_2 x_2^2$ where $u_1,u_2,u_3$ are units.  So $f \neq 0$.  We show that $f$ is irreducible and hence prime.  Suppose that $f = (Ax_1 + Bx_2)(Cx_1 + Dx_2)$.  Then $(u_1-AC)x_1^2 + (u_2 - (AD+BC))x_1x_2 + (u_3-BD)x_2^2 = 0$.  So $u_1 - AC, u_2-(AD+BC),u_2-BD \in M$.  Thus $AC$, $AD+BC$, $BD$ are units.  Hence $A,B,C,D$ are units.  But then $AD$ and $BC$ are units, so $AD+BC \in M$ since $D/M \cong \text{GF}(2)$.
\end{proof}

Let $(R,m)$ be a complete local CK domain with precisely three nonassociate atoms.  Then $R/m \cong \text{GF}(2)$ and $\text{dim}_{R/m} m/m^2 = 2$, so $R$ is a homomorphic image of a two-dimensional complete regular local ring $(D,M)$ with residue field $\text{GF}(2)$.  So if $\text{char} R = 2$, $D \cong \text{GF}(2) [[X,Y]]$ while if char $R =0$, $D \cong V[[X]]$ (with $M = (2,X)$) or $D \cong V[[X,Y]] / (g)$ (with $M = (\overline{X},\overline{Y})$) where $g = 2-h$ with $h \in (2,X,Y)^2$ and $(V,2V)$ is a complete DVR with residue field $\text{GF}(2)$.  In the first case char $R=2$, $R \cong \text{GF}(2)[[X,Y]] / (X^2 + XY + Y^2)$ (as $\text{GF}(2) [[X,Y]] / (X^2 + XY + Y^2)$ is such a ring and any such one is isomorphic to $\text{GF}(2) + \text{GF}(2^2)[[X]]X$).  In the second case where char $R = 0$ and $2 \notin m^2$, $R \cong V[[X]] / (u_1 \cdot 4 + u_2 \cdot 2X + u_3 X^2)$ where $u_1,u_2,u_3$ are units of $V[[X]]$.  In the third case where char $R=0$ and $2 \in m^2$, $R \cong V[[X,Y]] / (g,u_1X^2+u_2XY + u_3Y^2)$ where $g = 2-h$, $h \in (2,X,Y)^2$ and $u_1,u_2,u_3$ are units of $V[[X,Y]]$.

\section{CK Domains with $n$ Atoms}

Let $R$ be an integral domain.  We say that $R$ is \emph{primefree} if $R$ has no nonzero principal primes.  Let $\alpha$ be a possibly infinite cardinal number.  We say that \emph{$R$ has $\alpha$ atoms} if there is a set $A$ of atoms of $R$ with $|A|=\alpha$ such that every atom of $R$ is an associate of exactly one element of $A$.  In this section we are interested in local CK domains or more generally primefree CK domains with a prescribed number of atoms. However, we begin by noting that for an infinite cardinal number $\alpha$, there exists a one-dimensional local domain with $\alpha$ atoms.

\begin{ex} \label{AlphaInf}
(A one-dimensional local domain with $\alpha$ atoms for $\alpha$ infinite)  Let $(D,N)$ be a one-dimensional local domain that is not a DVR with $|D| \leq \alpha$.  Let $\{X_{\beta}\}_{\beta \in \Lambda}$ be a set of indeterminates with $|\Lambda|= \alpha$ and let $R = D(\{X_{\beta}\}_{\beta \in \Lambda})$ be the Nagata ring $D[\{X_{\beta}\}_{\beta \in \Lambda}]_{N[\{X_{\beta}\}_{\beta \in \Lambda}]}$.  Then $R$ is a one-dimensional local domain with maximal ideal $M=RN$.  As $|R|=\alpha$, $R$ has at most $\alpha$ atoms.  Let $m_1,m_2$ be part of a minimal basis for $N$.  Then for $\beta_1,\beta_2 \in \Lambda$, $\beta_1 \neq \beta_2$, $R(m_1+m_2X_{\beta_2},m_1+m_2X_{\beta_2}) = R(m_1,m_2)$.  So $m_1+m_2X_{\beta_1}$, $m_1 + m_2X_{\beta_2}$ is part of a minimal basis for $M$.  Thus $m_1 + m_2 X_{\beta_1}$ and $m_1 + m_2 X_{\beta_2}$ are nonassociate atoms of $R$.  Hence $R$ has $\alpha$ atoms.
\end{ex}

Thus we have the following question.

\begin{question}
For which natural numbers $n$, does there exist a local CK domain with $n$ atoms?
\end{question}

Now for $n=1$ we just have a DVR and by Theorem \ref{QLatom} (7) a local CK domain cannot have $2$ atoms. So suppose $n \geq 3$.  Cohen and Kaplansky \cite{CK} showed that if $n = (p^{nk}-1) / (p^n-1)$ where $p$ is prime and $m,k \geq 1$, there is a (complete) local CK domain $(R,M)$ with $M^2$ universal having $n$ atoms. We can construct $R$ as follows. Let $D$ be a DVR with residue field $\text{GF}(p^{nk})$ and take $R$ to be the pullback of the natural map $D \to \text{GF}(p^{nk})$ along $\text{GF}(p^m) \hookrightarrow \text{GF}(p^{mk})$. So such an $R$ can have characteristic $0$ or characteristic $p$. In the equicharacteristic complete case $R$ has the form $\text{GF}(p^m) + \text{GF}(p^{mk})[[X]]X$ (Example \ref{Wint}). This was generalized in \cite{AMO} (see Example \ref{Corrected}): $R = \text{GF}(p^m) + \text{GF}(p^{mk})[[X]]X^l$, $m,k,l \geq 1,$ is a complete local CK domain with $l((p^{mk}-1)/(p^m-1))p^{mk(l-1)}$ atoms.

Suppose that $n \geq 3$ is prime. Then there is a local CK domain $(R,M)$ with $n$ atoms if and only if $M^2$ is universal if and only if $n = (p^{mk}-1) / (p^m-1)$ for some prime $p$ and natural numbers $m,k$ with $k \geq 2.$ The first odd prime not of this form is $11$. So there is no local CK domain with $11$ atoms. (As we will see there are local CK domains with $n$ atoms for $n=3,4,\ldots,10$.) The odd primes less than $100$ not of this form are $11$, $19$, $23$, $29$, $37$, $41$, $43$, $47$, $59$, $67$, $71$, $73$, $79$, $83$, $89$, and $97$. Clark, Gosavi, and Pollack \cite{CGP} have noted that among the prime numbers, the set of primes of the form $(p^{mk}-1)/(p^m-1)$ has density $0$.

We have constructed three infinite families of positive characteristic CK domains

\begin{itemize}
    \item[(1)]  (Example \ref{Wint})  $\text{GF}(p^m) + \text{GF}(p^{mk})X + \text{GF}(p^{mkl})X^2$ \hspace*{\fill} $m,k,l \geq 1$
    \item[(2)]  (Example \ref{ExAtomsInM2})  $\text{GF}(p^m) + WX + GF(p^{3m})[[X]]X^2$ \hspace*{\fill} $m \geq 1$
    \item[(3)]  (Example \ref{Corrected})  $\text{GF}(p^m) + \text{GF}(p^{mk})[[X]]X^l$ \hspace*{\fill} $m,k,l \geq 1$
\end{itemize}

The following table gives the number of atoms, the number of atoms in $M^2$ (for (2) necessarily not in $M^3$), and the minimal power of $M$ which is universal for each family.

\begin{center}
  \begin{tabular}{ | c | c | c | c |}
    \hline
    \thead{Family} & Atoms & Atoms in $M^2$ & Universality \\ \hline
    1 & $\frac{(p^{mkl}-1)p^{mk(l-1)}}{p^m-1}$ & 0 & \makecell{$M^4$, $M^3 = M^2$ $\iff$ $l=1$ \\ $M \iff k=l=1$} \\ \hline
    2 & $\frac{p^m(p^{2m}+p^m+2)}{2}$ & $\frac{p^{2m}(p^m-1)}{2}$ & $M^4$ \\ \hline
    3 & $\frac{l(p^{mk}-1)p^{mk(l-1)}}{p^m-1}$ & 0 & \makecell{$M^3$, $M^2 \iff l=1$ \\ $M \iff k=1=1$} \\ \hline
  \end{tabular}
\end{center}

Suppose that we take $l=1$ in Family $1$ or $3$, so $M^2$ is universal. For $n < 100$, this gives local CK domains with $M^2$ universal for $n=1$, 3, 4, 5, 6, 7, 8, 9, 10, 12, 13, 14, 15, 17, 18, 20, 21, 24, 26, 28, 30, 31, 32, 33, 38, 40, 42, 44, 48, 50, 54, 57, 60, 62, 63, 65, 68, 72, 74, 80, 82, 84, 90, 91, 98.

For Family $1$ if we take $l > 1$, then $M^4$ is universal. For $n < 100$, we get local CK domains with $n=6$, 12, 20, 28, 30, 56, and $60$ atoms.

For Family $2$, $M^4$ is universal and for $n < 100$ we get local CK domains with $n=8,$ 21, 44, and $80$ atoms.

For Family $3$ if we take $l \geq 2$, then $M^3$ is universal. We get local CK domains with $n$ atoms for $l=2$: $n=4$, 6, 8, 10, 14, 16, 18, 22, 24, 26, 32, 34, 38, 46, 50, 54, 58, 62, 64, 72, 74, 82, 86, 94, and 98; $l=3$: $n=2$, 27, 48, and $75$; $l=4$: $n=32$.

Thus for $n < 100$, after deleting the primes $n$ with no local CK domain with $n$ atoms, it is unknown to us whether there exist local CK domains with $n$ atoms for $n=25$, 35, 36, 39, 45, 51, 52, 53, 55, 69, 70, 76, 77, 78, 81, 85, 87, 88, 92, 93, 95, 96, or 99.

So far most of our examples of local CK domains have been in characteristic $p$. Coykendall and Spicer \cite{CS} and Clark, Govasi, and Pollack \cite{CGP} gave some examples in characteristic $0$ from number theory. We begin with this following example.

\begin{ex} \label{Zpw}
 (\cite[Corollary 3.6]{CS}).  Let $d$ be a square free integer and $\Z[\omega]$ be the ring of integers of $\Q[\sqrt{d}]$. Let $p_1,\ldots,p_n$ be distinct primes that are inert in $\Z[\omega]$ and put $p=p_1 \cdots p_n$. Let $R = \Z[p \omega]_S$ where $S = R \backslash (p_1, p \omega) \cup \cdots \cup (p_n, p \omega)$. Then $R$ is a primefree CK domain with $n$ maximal ideals  $M_i = (p_i,p \omega)_S$, $i=1,\ldots,n$. Each $R_{M_i} = \Z[p \omega]_{(p_i,p \omega)}$ is a local CK domain with $p_i+1$ atoms. Thus $R$ has $\sum_{i=1}^n (p_i+1)$ atoms (see Theorem \ref{Lattice} (1)). Here one can show that $R_{M_i}$ is a local CK domain via Theorem \ref{CKequiv} ((1) $\iff$ (3)). It is easily checked that $M_{i_{M_i}}^2$ is universal. Since $M_{M_i}$ is doubly generated and $R_{M_i}$ has residue field $\text{GF}(p_i)$, $R_{M_i}$ has $(p_i^2-1)/(p_i-1) = p_i+1$ atoms.
\end{ex}

We next investigate the local CK domains with $n$ atoms for $n \leq 11$.

\begin{ex}
Local CK domain $(R,M)$ with $n \leq 11$ atoms.

\begin{itemize}
    \item[$n=1$:] $R$ is a DVR
    
    \item[$n=2$:] $R$ does not exist (Theorem \ref{QLatom} (7))
    
    \item[$n=3$:] $3$ is prime so $M^2$ is universal with $|\overline{R}|=2$ and $\text{dim}_{\overline{R}} M / M^2 = 2$. Examples include $\Z[2 \omega]_{(2,2 \omega)}$ with $2$ inert in $\Z[\omega]$ (say for example, $d=5$) (see Example \ref{Zpw}) in characteristic $0$ and the unique complete local example $\text{GF}(2)+\text{GF}(2^2)[[X]]X$ in positive characteristic.
    
    \item[$n=4$:] No atoms in $M^2$ and $M^3$ is universal, $|\overline{R}|=2$ or $3$ and $\text{dim}_{\overline{R}} M / M^2 = 2$.
    
    $M^2$ universal: Examples include $\Z[3 \omega]_{(3,3 \omega)}$ with $3$ inert in $\Z[\omega]$ (say for $d=5$) in characteristic $0$ and the unique complete local example $\text{GF}(3)+\text{GF}(3^2)[[X]]X$ in positive characteristic.
    
    $M^3$ universal: $\text{GF}(2) + \text{GF}(2)[[X]]X^2$
    
    \item[$n=5$:] $5$ is prime so $M^2$ is universal with $|\overline{R}| = 4$ and $\text{dim}_{\overline{R}} M / M^2 = 2$. Here $\text{GF}(2^2) + \text{GF}(2^4)[[X]]X$ is the unique complete local example in positive characteristic. For a characteristic $0$ example, let $D$ be a DVR of characteristic $0$ with residue field $\text{GF}(2^2)$ and take $R$ to be the pullback of the natural map $D \to \text{GF}(2^2)$ along $\text{GF}(2) \hookrightarrow \text{GF}(2^2)$.
    
    \item[$n=6$:] There are no atoms in $M^2$ and $M^5$ is universal, $|\overline{R}| \leq 5$, and $\text{dim}_{\overline{R}} M / M^2 = 2$.
    
    $M^2$ universal: $|\overline{R}|=5$. In positive characteristic $\text{GF}(5)+\text{GF}(5^2)[[X]]X$ is the unique example. In characteristic $0$ we can take $\Z[5 \omega]_{(5,5 \omega)}$ where $5$ is inert in $\Z[\omega]$ (say for $d=13$)
    
    $M^3$ universal: $\text{GF}(3) + \text{GF}(3)[[X]]X^2$
    
    $M^4$ universal: $\text{GF}(2) + \text{GF}(2)X + \text{GF}(2^2)[[X]]X^2$
    
    $M^5$ universal: no example known
    
    \item[$n=7$:] $7$ is prime so $M^2$ is universal with $|\overline{R}|=2$ and $\text{dim}_{\overline{R}} M  M^2 = 3$. Here $\text{GF}(2) + \text{GF}(2^3)[[X]]X$ is the unique complete local example in positive characteristic. For a characteristic $0$ example, take a DVR $D$ with residue field $\text{GF}(2^3)$ and take $R$ to be the pullback of the natural map $D \to \text{GF}(2^3)$ along $\text{GF}(2) \hookrightarrow \text{GF}(2^3)$.
    
    \item[$n=8$:] Here $M^7$ is universal, $|\overline{R}| \leq 7$ and $\text{dim}_{\overline{R}} M / M^2 = 2$ except for $|\overline{R}|=2$ and $\text{dim}_{\overline{R}} M / M^2 = 3$. Since there are at least $6$ atoms in $M \backslash M^2$, either all atoms are in $M \backslash M^2$ or there are $6$ atoms in $M \backslash M^2$ and $2$ in $M^2$. The case $|\overline{R}|=5$ is ruled out since this implies $|V| \geq 5$ so $|V|=8$ which gives $M^2$ universal, a contradiction.
    
    $M^2$ universal: $|\overline{R}|=7$ and $\text{dim}_{\overline{R}} M / M^2 = 2$. So $\text{GF}(7) + \text{GF}(7^2)[[X]]X$ is the unique complete local example in positive characteristic and $\Z[7 \omega]_{(7,7 \omega)}$ with $7$ inert in $\Z[\omega]$ (say $d=5$) is a characteristic $0$ example.
    
    $M^2$ not universal: Here $|V|=2$ or $4$. If $|V|=4$ all $8$ atoms are in $M \backslash M^2$. For $|V|=4$, $|\overline{R}| \leq 4$. If $|\overline{R}|=3$ or $4$, $\text{dim}_{\overline{R}} M / M^2 = 2$.
    
    $M^3$ universal example: $\text{GF}(2^2) + \text{GF}(2^2)[[X]]X^2$ (no atoms in $M^2$)
    
    $M^4$ universal example: $\text{GF}(2) + WX + \text{GF}(2^3)[[X]]X^2$ (Example \ref{8atom}) (6 atoms in $M \backslash M^2$, 2 atoms in $M^2$).
    
    \item[$n=9$:] Here $|V|=9$ in which case $M^2$ is universal or $|V|=3$. There are either $9$ atoms in $M \backslash M^2$ or $6$ atoms in $M \backslash M^2$ and $3$ in $M^2$.
    
    $M^2$ universal: Here $|\overline{R}|=8$ and $\text{dim}_{\overline{R}} M / M^2 = 2$. We have that $\text{GF}(2^3) + \text{GF}(2^6)[[X]]X$ is the unique complete local example in positive characteristic. In characteristic $0$ we can take a DVR with residue field $\text{GF}(2^6)$ and take $R$ to be the pullback of the natural map $D \to \text{GF}(2^6)$ along $\text{GF}(2^3) \hookrightarrow \text{GF}(2^6)$. 
    
    $M^2$ is not universal: So $|V|=3$. So $|\overline{R}|=2$ or $3$. And we have that  $\text{dim}_{\overline{R}} M / M^2 = 2$ unless  $|\overline{R}|=2$ and $\text{dim}_{\overline{R}} M / M^2 = 3$. Suppose $|\overline{R}|=3$. Then we cannot have an atom in $M^2$ since an atom in $M^2$ would give at least $4 \cdot 3$ atoms in $M \backslash M^2$. Suppose $|\overline{R}|=2$. Here either all $9$ atoms are in $M \backslash M^2$ or $6$ are in $M \backslash M^2$ and $3$ in $M^2$. If $\text{dim}_{\overline{R}} M / M^2 = 3$, all atoms are in $M \backslash M^2$. Cohen and Kaplansky claimed that for $n=9$, we cannot have atoms in $M^2$. We have been unable to verify this.
    
    \item[$n=10$:] Here $|V|=2, 5$, or 10.  If $|V|=10$, $M^2$ is universal. Here $|\overline{R}|=9$ and also $\text{dim}_{\overline{R}} M / M^2 = 2$. So in the positive characteristic case $\text{GF}(3^2) + \text{GF}(3^4)[[X]]X$ is the unique complete local example. A characteristic $0$ example can be obtained via pullbacks. If $|V|=5$, then all ten atoms are in $M \backslash M^2$. The only remaining case $|V|=2$ forces $|\overline{R}|=2$ and $\text{dim}_{\overline{R}} M / M^2 = 2$ or $3$. If $\text{dim}_{\overline{R}} M / M^2 = 3$, there are at least $7$ atoms in $M \backslash M^2$ and hence all atoms are in $M \backslash M^2$. Suppose that $\text{dim}_{\overline{R}} M / M^2 = 2$. If there is an atom in $M^2$, there are at least $6$ atoms in $M \backslash M^2$. Thus either all $10$ atoms are in $M \backslash M^2$, $8$ are in $M \backslash M^2$ and $2$ in $M^2$, or $6$ in $M \backslash M^2$ and $4$ in $M^2$.
    
    \item[$n=11$:] $11$ is prime and not of the form $(p^{nk}-1) / (p^n-1)$. So an example does not exist.
\end{itemize}
\end{ex}

We end by considering the case of nonlocal CK domains. Using Example \ref{Zpw}, Coykendall and Spicer \cite{CS} showed that for $n=\sum_{i=1}^m (p_i+1)$ for distinct primes $p_1, \ldots, p_n$ there is a primefree CK domain with $n$ atoms. Assuming a variant of the Goldbach Conjecture (each even number $n \geq 6$ is a sum of two distinct primes), they show that for $n \geq 3$, there is a primefree CK domain with $3$ or less maximal ideals with $n$ atoms. Then Clark, Gosavi, and Pollack \cite[Theorem 1.4]{CGP} showed using a generalization of Bertrand's Postulate that each $n \geq 6$ can be written as $\sum_{i=1}^m (p_i+1)$ for distinct primes $p_1,\cdots,p_n$. Using Theorem \ref{Lattice} (2), they obtained a number of interesting results concerning primefree CK domains with $n$ atoms and $m$ maximal ideals. We list some of their results.

\begin{thm}
Let $m$ and $n$ be positive integers.

\begin{itemize}
\item[(1)] \cite[Theorem 1.6]{CGP} If $n \geq 10$ is even (resp., $n \geq 13$ is odd) and $m \in [3,\frac{n}{3}]$ (resp., $m \in [4,\frac{n}{3}])$, there is a primefree CK domain of characteristic $0$ with $n$ atoms and $m$ maximal ideals.

\item[(2)] \cite[Theorem 1.11]{CGP} Let $q$ be a prime power. If $q$ is even (resp., $q$ is odd), then for all $n \geq 2q^2-q$ (resp., $n \geq 2q^2 - q + 1$), there is a primefree CK domain with $n$ atoms that is a $\text{GF}(q)$-algebra.
\end{itemize}
\end{thm}

We end with the following result that shows that for CK domains we can usually reduce to the complete local domain case. Here the first statement is well known while the second is just a restatement of a result of Clark, Gosavi, and Pollack \cite[Theorem 1.10]{CGP}.

\begin{thm} \label{Lattice}
\begin{itemize}
    \item[(1)] Let $R$ be a CK domain with maximal ideals $M_1,\ldots,M_n$. Then the map
    
    $$L(R) \to L(\widehat{R_{M_1}}) \times \cdots \times L(\widehat{R_{M_n}})$$
    
    given by
    
    $$A \to (\widehat{R_{M_1}}A,\ldots,\widehat{R_{M_n}}A)$$
    
    is a multiplicative lattice isomorphism that induces an order preserving monoid isomorphism of the positive cones of the corresponding groups of divisibility
    
    \begin{align*}
        G_+(R) \to G_+(\widehat{R_{M_1}}) \times \cdots \times G_+(\widehat{R_{M_n}}) \\
        aU(R) \to (aU(\widehat{R_{M_1}}),\ldots,aU(\widehat{R_{M_n}}))
    \end{align*}
    
    where $\; \; \; \widehat{} \; \; $ denotes the appropriate $M$-adic completion.
    
    \item[(2)] Let $(R_1,m_1),\ldots,(R_n,m_n)$ be local CK domains with finite residue fields such that either (1) each char $R_i=0$ or (2) each $R_i$ is an $F$-algebra where $F$ is a finite field. Then there exists a CK domain $R$ with maximal ideals $M_1,\ldots,M_n$ such that either (1) char $R=0$ or (2) $R$ is an $F$-algebra, each $R_i / m_i \cong R_{M_i} / M_{i_{M_i}}$ and $\widehat{R_i} \cong \widehat{R_{M_i}}$, where $\; \; \widehat{} \; $ denotes the appropriate $M$-adic completion. Thus if each $R_i$ has $l_i$ atoms, $R$ has $l_1 + \cdots + l_n$ atoms and if no $R_i$ is a DVR, $R$ is primefree.    
\end{itemize}
\end{thm}

\begin{proof}
(1) Let $A \neq 0$ be an ideal of $R$. Then $A = (A_{M_1} \cap R) \cap \cdots \cap (A_{M_n} \cap R) = (A_{M_1} \cap R) \cdots (A_{M_n} \cap R)$ where each $A_{M_i} \cap R = R$ or is $M_i$-primary. Moreover, $A$ is principal if and only if each $A_{M_i} \cap R$ is principal. This induces a multiplication lattice homomorphism $L(R) \to L(R_{M_1}) \times \cdots \times L(R_{M_n})$ by $A \mapsto (A_{M_1},\ldots,A_{M_n})$. Also since each nonzero ideal of $R_{M_i}$ (resp., $\widehat{R_{M_i}}$) is $M_{i_{M_i}}$-primary (resp., $\widehat{M_{i_{M_i}}}$-primary), the map $L(R_{M_i}) \to L(\widehat{R_{M_i}})$ given by $A \to \widehat{R_{M_i}} A$ is a multiplicative lattice isomorphism. Clearly both maps preserve principal ideals. There the map $A \to (\widehat{R_{M_1}}A, \ldots,\widehat{R_{M_n}}A)$ is a multiplicative lattice map that preserves principal ideals.

(2) The proof of \cite[Theorem 1.10]{CGP} showed that if $R_1,\ldots,R_n$ are any primefree local CK domains with either (1) each char $R_i=0$ or (2) each $R_i$ is an $F$-algebra where $F$ is a finite field, then there is a primefree CK domain $R$ with (1) either char $R=0$ or (2) $R$ is an $F$-algebra such that $R_i / m_i \cong R_{M_i} / M_{M_i} (\cong R/M_i)$ and $\widehat{R_i} \cong \widehat{R_{M_i}}$. The condition that the $R_i$ were primefree gives that each residue field $R_i / m_i$ is finite. The same proof carries through if we allow $R_i$ to be a DVR as long as $R_i / m_i$ is finite.
\end{proof}

\bibliographystyle{elsarticle-harv}
\bibliography{Paper.bib}

\begin{thebibliography}{14}
\expandafter\ifx\csname natexlab\endcsname\relax\def\natexlab#1{#1}\fi
\expandafter\ifx\csname url\endcsname\relax
  \def\url#1{\texttt{#1}}\fi
\expandafter\ifx\csname urlprefix\endcsname\relax\def\urlprefix{URL }\fi

\bibitem[{Anderson(1978)}]{A}
Anderson, D.~D., 1978. Some finiteness conditions on a commutative ring.
  Houston J. Math 4, 289 -- 299.

\bibitem[{Anderson and Anderson(1992)}]{AA}
Anderson, D.~D., Anderson, D.~F., 1992. Elasticity of factorizations in
  integral domains. J. Pure Appl. Algebra 80~(3), 217 -- 235.

\bibitem[{Anderson et~al.(1990)Anderson, Anderson, and Zafrullah}]{AAZ}
Anderson, D.~D., Anderson, D.~F., Zafrullah, M., 1990. Factorization in
  integral domains. J. Pure Appl. Algebra 69~(1), 1 -- 19.

\bibitem[{Anderson and Juett(2015)}]{AJ}
Anderson, D.~D., Juett, J.~R., 2015. Long length functions. J. Algebra 426, 327
  -- 343.

\bibitem[{Anderson and Mott(1992)}]{AMO}
Anderson, D.~D., Mott, J.~L., 1992. Cohen-{K}aplansky domains: Integral domains
  with a finite number of irreducible elements. J. Algebra 148~(1), 17 -- 41.

\bibitem[{Anderson and Mullins(1996)}]{AMU}
Anderson, D.~D., Mullins, B., 1996. Finite factorization domains. Proc. Amer.
  Math. Soc. 124~(2), 389 -- 396.

\bibitem[{Clark et~al.(2017)Clark, Gosavi, and Pollack}]{CGP}
Clark, P.~L., Gosavi, S., Pollack, P., 2017. The number of atoms in a primefree
  atomic domain. Comm. Algebra 45, 5431 -- 5442.

\bibitem[{Cohen and Kaplansky(1946)}]{CK}
Cohen, I.~S., Kaplansky, I., 1946. Rings with a finite number of primes. {I}.
  Trans. Amer. Math. Soc. 60~(3), 468 -- 477.

\bibitem[{Coykendall et~al.(1999)Coykendall, Dobbs, and Mullins}]{CDM}
Coykendall, J., Dobbs, D.~E., Mullins, B., 1999. On integral domains with no
  atoms. Comm. Algebra 27~(12), 5813 -- 5831.

\bibitem[{Coykendall and Spicer(2012)}]{CS}
Coykendall, J., Spicer, C., 2012. Cohen-{K}aplansky domains and the {G}oldbach
  conjecture. Proc. Amer. Math. Soc. 140, 2227 -- 2233.

\bibitem[{Glastad and Mott(1982)}]{GM}
Glastad, B., Mott, J.~L., 1982. Finitely generated groups of divisibility.
  Contemp. Math 8, 231 -- 247.

\bibitem[{Grams(1974)}]{G}
Grams, A., 1974. Atomic rings and the ascending chain condition for principal
  ideals. Math. Proc. Cambridge Philos. Soc. 75~(3), 321 – 329.

\bibitem[{Houston et~al.(1988)Houston, Lucas, and Viswanathan}]{HLV}
Houston, E.~G., Lucas, T.~G., Viswanathan, T.~M., 1988. Primary decomposition
  of divisorial ideals in {M}ori domains. J. Algebra 117, 327 -- 342.

\bibitem[{McCoy(1957)}]{M}
McCoy, N., 1957. A note on finite unions of ideals and subgroups. Proc. Amer.
  Math. Soc. 8, 633 -- 637.

\end{thebibliography}

\end{document}